\documentclass[11pt]{article}
\usepackage{inputenc}
\usepackage{enumerate}
\usepackage{amssymb}
\usepackage{amsmath}
\usepackage{amsthm}
\usepackage{tikz-cd}
\usepackage[square,sort,comma,numbers]{natbib}
\usepackage{mathpazo}

\usepackage{extarrows}
\usepackage{color}
\usepackage{setspace}
\usepackage{hyperref}
\usepackage{graphicx, subfig}
\usepackage{natbib}
\usepackage{setspace}
\theoremstyle{definition}
\newtheorem{Def}{Definition}[section]
\newtheorem{Exa}[Def]{Examples}
\newtheorem{Rem}[Def]{Remark}
\theoremstyle{plain}
\newtheorem{Thm}[Def]{Theorem}
\newtheorem{Lem}[Def]{Lemma}
\newtheorem{Pro}[Def]{Proposition}
\newtheorem{Cor}[Def]{Corollary}

\newtheorem*{CNcon}{The coarse Novikov conjecture}
\usepackage[left=2.6cm, right=2.6cm, top=3cm, bottom=3cm]{geometry}
\usepackage[all]{xy}
\usepackage{bm}

\def\Box{{\bf Box}}
\def\IN{\mathbb N}
\def\IR{\mathbb R}
\def\IE{\mathbb E}
\def\IC{\mathbb C}
\def\IL{\mathbb L}
\def\IA{\mathbb A}
\def\A{\mathcal A}
\def\C{\mathcal C}
\def\B{\mathcal B}

\def\O{\mathcal O}

\def\K{\mathcal K}

\def\Q{\mathcal Q}
\def\H{\mathcal H}
\def\S{\mathcal S}
\def\N{\mathcal N}

\def\supp{\textup{supp}}
\def\prop{\textup{Prop}}
\def\diam{\textup{diam}}
\def\Cl{\textup{Cliff}_{\IC}}
\def\wox{\widehat{\otimes}}
\def\ox{\otimes}
\def\wh{\widehat}
\def\ox{\otimes}
\def\ICu{\mathbb C_{u,\infty}}
\def\ICuL{\mathbb C_{L,u,\infty}}
\def\Cau{C^*_{u,\infty}}
\def\CauL{C^*_{L,u,\infty}}
\def\PdX{(P_d(X_n))_{n\in\IN}}

\author{Liang Guo\and Zheng Luo\and Qin Wang\and Yazhou Zhang}
\title{A Bott periodicity theorem for $\ell^p$-spaces and the coarse Novikov conjecture at infinity\footnote{Supported in part by NSFC (No. 11831006, 12171156).}}
\begin{document}
\maketitle

\begin{abstract}
We formulate and prove a Bott periodicity theorem for an $\ell^p$-space ($1\leq p<\infty$). For a proper metric space $X$ with bounded geometry, we introduce a version of $K$-homology at infinity, denoted by $K_*^{\infty}(X)$, and the Roe algebra at infinity, denoted by $C^*_{\infty}(X)$. Then the coarse assembly map descends to a map from $\lim_{d\to\infty}K_*^{\infty}(P_d(X))$ to $K_*(C^*_{\infty}(X))$, called the coarse assembly map at infinity. We show that to prove the coarse Novikov conjecture, it suffices to prove the coarse assembly map at infinity is an injection. As a result, we show that the coarse Novikov conjecture holds for any metric space with bounded geometry which admits a fibred coarse embedding into an $\ell^p$-space. These include all box spaces of a residually finite hyperbolic group and a large class of warped cones of a compact space with an action by a hyperbolic group.

\end{abstract}

\begin{spacing}{0.5}
\tableofcontents
\end{spacing}

\section{Introduction}

The Novikov conjecture claims that the higher signatures of closed oriented smooth manifolds are invariant under orientation preserving homotopy equivalences. It is implied by the injectivity of the Baum-Connes assembly map \cite{BCH1994}. The coarse Novikov conjecture, as a geometric analogue of the Novikov conjecture, is the injectivity part of the coarse Baum--Connes conjecture \cite{HR1993,Roe1993}. It offers a technique to determine when an elliptic operators higher index on a noncompact complete Riemannian manifold is nonzero. A detailed survey on the coarse Novikov conjecture can be found in \cite{Yu2006}.

In \cite{CWY2013}, X.~Chen, Q.~Wang and G.~Yu generalized Gromov's notion of coarse embedding into Hilbert space and introduced the notion of fibred coarse embedding (see Definition \ref{FCE}) and showed that if a discrete metric space $X$ with bounded geometry admits a fibred coarse embedding into Hilbert space, then the maximal coarse Baum--Connes conjecture holds for $X$. Later in \cite{Martin2014}, M.~Finn-Sell showed that the coarse Novikov conjecture also holds for $X$ in such a case.

Our main theorem in this paper states as follows:

\begin{Thm}\label{main result}\emph{[Theorem \ref{main theorem}]}
Let $X$ be a metric space with bounded geometry. If $X$ admits a fibred coarse embedding into an $\ell^p$-space for some $p\in [1,\infty)$, then the coarse Novikov conjecture holds for $X$, i.e. the coarse assembly map
$$\mu:\lim_{d\to\infty}K_*(P_d(X))\to K_*(C^*(X))$$
is injective.
\end{Thm}

Warped cones over actions of hyperbolic groups and box spaces of residually finite hyperbolic groups are typical examples to which our main result appies. A sequence of nested normal subgroups $\{\Gamma_i\}$ of a residually finite group $\Gamma$ is said to be a filtration if the sequence has trivial intersection. The box space of $\Gamma$ associated to a filtration $\{\Gamma_i\}$, denoted by $\Box_{\{\Gamma_i\}}\Gamma$, is the coarse disjoint union $\bigsqcup\Gamma/\Gamma_i$, where each quotient is equipped with the word length metric. Warped cones are metric spaces introduced by J.~Roe \cite{Roe2005} from discrete group actions on compact metric spaces. Typically, let $\Gamma$ be a finite generated group, $M$ a compact Riemannian manifold. The warped cone for $(\Gamma,M)$, denoted by $\mathcal{O}_{\Gamma}(M)$, is constructed from the action $\Gamma$ on $M$ by diffeomorphisms. It is a metric space whose underlying topological space is the open cone $\mathcal{O}M$ on $M$, but whose coarse geometry produces large groups of translations. The second section will have a comprehensive introduction for warped cones.

In \cite{SW2021}, D.~Sawicki and J.~Wu found a relationship between fibred coarse embeddability of warped cones and existence of proper affine isometric actions of discrete groups on Banach spaces. They proved that if $(\Gamma, M)$ is free and admits a linearization by unitary operators in a Hilbert space $\H$ and $\Gamma$ admits a proper affine isometric action on $\H$ (i.e., $\Gamma$ is a-T-menable), then $\O_{\Gamma}(M)$ admits a fibred coarse embedding into Hilbert space. The result still holds if we replace the Hilbert space with an $\ell^p$-space. The same question is also considered by Q.~Wang and Z.~Wang in \cite{WW2017}. They provided a different proof for the case of $\ell^p$ spaces and replace the condition $\Gamma\curvearrowright M$ is free by that $M$ is required to contain a dense and free orbit of $\Gamma$ in their paper.

As corollaries, we have the following results:

\begin{Cor}
Let $\Gamma$ be a finite generated discrete group and let $M$ be a compact manifold. Assume that the action of $\Gamma$ on $M$ is linearisable in an $\ell^p$-space and $\Gamma\curvearrowright M$ is free (or $M$ contains a dense free orbit). If $\Gamma$ admits a proper affine isometric action on an $\ell^p$-space, then the coarse Novikov conjecture holds for the warped cone $\mathcal{O}_{\Gamma}(M)$.
\end{Cor}

Combining the result in \cite{Yu2005,Saw2018,CWW2013}, we also have

\begin{Cor}
Let $\Gamma$ be a finitely generated residually finite hyperbolic group, and $G=\widehat{\Gamma} ((\Gamma_n))$ the corresponding profinite completion. Then the coarse Novikov conjecture hold for $\mathcal{O}_\Gamma (G)$ and any box apace $\Box_{\{\Gamma_n\}}\Gamma$.
\end{Cor}

We now give a short explaination about serveral key ingredients in our proof. We briefly recall the proof of the maximal coarse Baum-Connes conjecture for a metric space which admits a fibred coarse embedding into a Hilbert space in \cite{CWY2013}. 
The authors firstly reduced the issue to proving the maximal coarse Baum-Connes conjecture at infinity for the coarse disjoint unions of finite metric spaces. By using a geometric Dirac-dual-Dirac argument, they further reduced the problem to a twisted version of coarse Baum-Connes conjecture at infinity by using a $C^*$-algebra constructed in \cite{HKT1998}. Then the theorem follows by using a cutting and pasting argument of G.~Yu introduced in the Section 6 of \cite{Yu2000}. We try to follow the outline of their proof, however there are servarl problems.

For the first step, we need to use a Mayer-Vietoris argument and Five Lemma to reduce the problem to 'infinity'. First of all, a coarse Mayer-Vietoris argument in \cite{HRY1993} enables us to just take into consideration of the coarse Baum-Connes conjecture for coarse disjoint unions of finite subsets. For the case when $X=\bigsqcup_{n\in\IN}X_n$ where each $X_n$ is finite, one can use the following diagram, see \cite{CWY2013}:
$$\xymatrix{0\ar[d]&0\ar[d]\\\lim\limits_{d\to\infty}K_*(P_d(X_{N_d}))\oplus\bigoplus_{n=N_d}^{\infty}K_*(P_d(X_n))\ar[r]^{\qquad\qquad\qquad\cong}\ar[d]&K_*(\mathcal{K})\ar[d]\\
\lim\limits_{d\to\infty}K_*(P_d(X))\ar[r]^{\mu}\ar[d]&K_*(C^*_{max}(X))\ar[d]^{\Phi_*}\\
\lim\limits_{d\to\infty}\frac{\prod_{n\in\IN}K_*(P_d(X_n))}{\bigoplus_{n\in\IN}K_*(P_d(X_n))}\ar[r]^{\mu_{\infty}}\ar[d]&K_*(C^*_{max,u,\infty}(P_d(X_n)))\ar[d]\\0&0}.$$
The vertical sequence on the right side is exact only when we take the maximal norm and the top horizantal map is an isomorphism. Then it suffices to prove the assembly map at infinity $\mu_{\infty}$ is an isomorphism.

However, when we consider the coarse Novikov conjecture for certain spaces, the coarse Mayer-Vietories argument above does not work anymore as the Five Lemma fails if we only have the injectivity part of the assembly map $\mu$.
The argument holds only for a coarse disjoint union of a sequence of finite metric spaces $X=\bigsqcup_{n\in\IN}X_n$. One can still reduce the coarse Novikov conjecture for $X$ to 'infinity' by using the following diagram, see \cite{RYI2012,RYII2012}.
$$\xymatrix{0\ar[d]&0\ar[d]\\\lim\limits_{d\to\infty}K_*(P_d(X_{N_d}))\oplus\bigoplus_{n=N_d}^{\infty}K_*(P_d(X_n))\ar[r]^{\qquad\qquad\qquad\cong}\ar[d]&K_*(\mathcal{K})\ar[d]^{i_*}\\
\lim\limits_{d\to\infty}K_*(P_d(X))\ar[r]^{\mu}\ar[d]&K_*(C^*(X))\ar[d]^{\Phi_*}\\
\lim\limits_{d\to\infty}\frac{\prod_{n\in\IN}K_*(P_d(X_n))}{\bigoplus_{n\in\IN}K_*(P_d(X_n))}\ar[r]^{\mu_{\infty}}\ar[d]&K_*(C^*_{u,\infty}(P_d(X_n)))\\0&}.$$
Notice that $i_*:K_*(\K)\to K_*(C^*(X))$ is injective only when $X$ is a coarse disjoint union of finite spaces, see \cite{OY2009}. By using a diagram chasing argument, to prove $\mu$ is injective, it suffices to prove $\mu_{\infty}$ is injective.

For the general case (\emph{especially the case when $X$ can not be seen as a coarse disjoint union of finite metric spaces}), in order to make the whole progress work well, we introduce a notion of $K$-homology at infinity for a locally compact Hausdorff space and a notion of Roe algebra at infinity for a proper metric space in Section 4 and construct the following diagram:
$$\xymatrix{\lim\limits_{d\to\infty}K_*(P_d(X))\ar[r]^{\pi_*}\ar[d]_{\mu}&\lim\limits_{d\to\infty}K_*^{\infty}(P_d(X))\ar[d]^{\mu_{\infty}}\\
K_*(C^*(P_d(X)))\ar[r]^{\Phi_*}&K_*(C^*_{\infty}(P_d(X)))}.$$
A metric space with bounded geometry is said to admit an \emph{infinite coarse component} if there exists $R>0$ such that the Rips complex $P_R(X)$ has an unbounded component. It is an opposite statement to the coarse disjoint union (see Proposition \ref{fcc and cdu}). We prove that $\pi_*$ is an injection if $X$ admits an infinite coarse component. Thus we can still reduce the Novikov conjecture for $X$ to proving $\mu_{\infty}$ is injective in this case. Moreover, we show in Section 5 that our construction coincides with the construction in \cite{CWY2013} for the coarse disjoint union case. This means that we can always reduce the coarse Novikov conjecture to 'infinity' by using our construction (see Theorem \ref{CNC at inf}).

For the second step, we need to construct a twisted version of the coarse Baum-Connes conjecture by using an algebra associated with $B=\ell^p(\IN,\IR)$ for $p\in [1,\infty)$. As $B$ is of property (H) introduced by G.~Kasparov and G.~Yu in \cite{KY2012}, we wanted to use the $C^*$-algebra $\Q((\A_n\wox\K)_{n\in\IN})$ introduced in \cite{KY2012,CWY2015} to solve this problem. For any $a\in\Q((\A_n\wox\K)_{n\in\IN})$, we shall need a $C^*$-isomorphism $t^*:\Q((\A_n\wox\K)_{n\in\IN})\to\Q((\A_n\wox\K)_{n\in\IN})$ associated to an affine isometry $t$ on $B$ to translate the support of $a$. However, $\Q((\A_n\wox\K)_{n\in\IN})$ is built by using a dense subspace of $B$ which may not be invarinat under $t$. To solve this problem, we construct a new algebra $\A(B)$ for an $\ell^p$-space $B=\ell^p(\IN,\IR)$ in Section 3. Our construction is inspired from the paper of S.~Gong, J.~Wu and G.~Yu \cite{GWY2018} for the case when $p=2$. We construct a Bott homomorphism $\beta_{x_0}:\S\to\A(B)$ associated with a base point $x_0\in B$ by using the $p/2$-H\"older extension of the Mazur map introduced by E.~Odell and T.~Schlumprecht in \cite{OS1994} and extended by Q.~Cheng in \cite{Cheng2016}. We show that the Bott homomorphism induces an isomorphism on $K$-theory which does not depend on the choice of the base point, i.e., $(\beta_{x_0})_*:K_*(\S)\to K_*(\A(B))$ is an isomorphism for any $x_0\in B$ (see Theorem \ref{K of A(B)}).

Actually, it has been pointed out in G.~Kasparov and G.~Yu's paper \cite{KY2012} that the Bott periodicity theorem holds for $\ell^p$-spaces. Indeed, the Mazur map can be extended to a homeomorphism between $\ell^p(\IN,\IR)$ and $\ell^2(\IN,\IR)$ and the Bott periodicity for $\ell^2(\IN,\IR)$ has been proved in \cite{HKT1998}. However, the constructions in \cite{HKT1998} and \cite{KY2012} also rely on a dense subspace of $\ell^p(\IN,\IR)$ which is sightly different from ours.

This paper is organized as follows.
In Section 2, we briefly recall the Roe algebra and Yu's localization algebra, and states the coarse Novikov conjecture. After that we go over the concept of warped metric and warped cones introduced by John Roe \cite{Roe2005}. 
In Section 3, we introduce an algebra $\A(B)$ associated with $B=\ell^p(\IN,\IR)$ and calculate its $K$-theory. It will play an important role in the following proof.
In Section 4, we introduce $K$-homology at infinity, Roe algebra at infinity and the coarse Novikov conjecture at infinity for a proper metric space.
In Section 5. we compare the case of coarse disjoint union and the case of spaces which admit an infinite coarse component, and reduce both cases to the coarse Novikov conjecture at infinity.
In Section 6, we define twisted algebras at infinity and compute their $K$-theories by using a cutting and pasting technique introduced by G.~Yu in \cite{Yu2000}.
In Section 7, we define the Bott map and complete the proof.

\section{The coarse Novikov conjecture}

First of all, let us recall the definition of the Roe algebra and the coarse Baum-Connes assembly map for a proper metric space (cf. \cite{BCH1994,Roe1993,Roe1996}) in this section.

Assume that $(X,d)$ is a proper metric space. For $\delta>0$, a $\delta$-net of $X$ is a discrete subset of $X_{\delta}\subset X$ such that there exists $r>0$ such that $d(x,x')>r$ for all distinct $x,x'\in X_{\delta}$ and $\bigcup_{x\in X_{\delta}}B(x,\delta)=X$. Moreover, $X$ is said to have bounded geometry if $X$ contains a net $X_{\delta}$ with bounded geometry for some $\delta>0$, i.e., for any $r>0$ there exists $N>0$ such that any ball of radius $r$ in $X_{\delta}$ contains at most $N$ elements.

We denote $C_0(X)$ to be the $C^*$-algebra of all continuous functions on $X$ which vanish at infinity. A separable infinite-dimensional Hilbert space $H_X$ is said to be an $X$-module if $H_X$ is equipped with a non-degenerate $*$-representation $\pi:C_0(X)\to\B(H_X)$. Moreover, $H_X$ is said to be ample if no non-zero element of $C_0(X)$ acts as a compact operator.

\begin{Def}[\cite{Roe1993}]
Let $H_X$ be an ample $X$-module and $T\in\B(H_X)$.
\begin{itemize}
\item[(1)]$T$ is said to be locally compact if $fT$ and $Tf$ are compact operators for any $f\in C_0(X)$;
\item[(2)]The support of $T$, denoted by $\supp(T)$, is defined to be the set of all points $(x,y)\in X\times X$ such that for all $f,g\in C_0(X)$ with $f(x)\ne 0$ and $g(y)\ne0$, we have that $fTg\ne0$;
\item[(3)]The propagation of $T$ is defined to be
$$\prop(T)=\sup\{d(x,y)\mid (x,y)\in\supp(T)\}\in [0,\infty].$$
Moreover, $T$ is said to have finite propagation if $\prop(T)<\infty$.
\item[(4)]The algebraic Roe algebra of $X$, denote by $\IC[X,H_X]$ (or simply $\IC[X]$), is defined to be the $*$-algebra of all finite propagation locally compact operators on $H_X$.
\item[(5)]The Roe algebra of $X$, denoted by $C^*(X,H_X)$ (or simply $C^*(X)$), is the norm closure of $\IC[X]$ in $\B(H_X)$.
\end{itemize}\end{Def}

Note that $C^*(X,H_X)$ does not depend on the choice of $X$-module $H_X$ up to a non-canonical $*$-isomorphism \cite{HRY1993}. Hence, it is convenient to discuss the following specific $X$-module. Let $Z\subseteq X$ be a countable dense subset of $X$ and $H$ be a fixed infinite dimensional Hilbert space. We denote $\K(H)$ to be the set of all compact operators on $H$. We define $H_X=\ell^2(Z)\ox H$ to be an ample $X$-module equipped with the pointwise multiplication representation of $C_0(X)$ on $\ell^2(Z)$. It is easy to check that $H_X$ is ample.

\begin{Def}\label{matrix Roe algebra}
Define $\IC_f[X]$ to be the set of all bounded functions $T:Z\times Z\to\K(H)$ satisfying the following conditions:
\begin{itemize}
\item[(1)]for any bounded subset $B\subset X$, the set
$$\#\{(x,y)\in(B\times B)\cap(X\times X)\mid T(x,y)\ne 0\}<\infty;$$
\item[(2)]there exists $L>0$ such that
$$\#\{y\in Z|T(x,y)\ne0\}<L,\qquad\#\{y\in Z\mid T(y,x)\ne 0\}<L$$
for all $x\in Z$;
\item[(3)]there exists $R>0$ such that $T(x,y)=0$ whenever $d(x,y)>R$ for $x,y\in Z$.
\end{itemize}\end{Def}

Notice that $\IC_f[X]$ can be viewed as a dense $*$-subalgebra of $\IC[X,H_X]$ in $C^*(X)$, where the algebraic operation of $\IC_f[X]$ is given by viewing $T\in \IC_f[X]$ as $Z$-by-$Z$ matrix. In this sequel, we will use $\IC_f[X]$ to replace $\IC[X]$ to define the Roe algebra of $X$.

We next recall the assembly map $\mu$ for the Roe algebras. Let $X$ be a proper metric space. Recall that the $K$-homology groups $K_i(X)=KK_i(C_0(X),\IC)$ $(i=0,1)$ are generated by certain cycles modulo certain equivalence relations

\begin{itemize}
\item[(1)]a cycle for $K_0(X)$ is a pair $(H_X,F)$, where $H_X$ is an $X$-module and $F$ is a bounded linear
operator acting on $H_X$ such that $F^*F-I$ and $FF^*-I$ are locally compact, and $\varphi F - F\varphi$
is compact for all $\varphi\in C_0(X)$;
\item[(2)] a cycle for $K_1(X)$ is a pair $(H_X,F)$, where $H_X$ is an $X$-module and $F$ is a self-adjoint
operator acting on $H_X$ such that $F^2-I$ is locally compact, and $\varphi F- F\varphi$ is compact for all
$\varphi\in C_0(X)$.
\end{itemize}

Let $(H_X,F)$ represent a cycle in $K_0(X)$. For any $R> 0$, one can always take a locally finite, uniformly bounded open cover $\{U_i\}_{i\in I}$ of $X$ such that the diameter of each $U_i$ is no more than $R$. Let $\{\phi_i\}_{i\in I}$ be a continuous partition of unity subordinate to the open cover $\{U_i\}_{i\in I}$. Define
$$F'=\sum_{i}\phi^{\frac 12}_iF\phi^{\frac 12}_i,$$
where the sum converges in the strong operator topology. It is not hard to see that $(H_X,F)$ and $(H_X,F')$ are equivalent via $(H_X,(1-t)F+tF')$, where $t\in[0,1]$. Note that both $F'$ and $F'^2-1$ have finite propagation, so $F'$ is a multiplier of $C^*(X)$ and $F'$ is invertible modulo $C^*(X)$. Hence $F'$ gives rise to an element, denoted by $\partial([F'])$ in $K_0(C^*(X))$, where
$$\partial:K_1(M(C^*(X))/C^*(X))\to K_0(C^*(X))$$
is the boundary map of $K$-theory, and $M(C^*(X))$ is the multiplier algebra of $C^*(X)$. We define the index of $(H_X,F)$ to be $\partial([F'])$. Similarly, we can define the index map from $K_1(X)$ to $K_1(C^*(X))$.

Now we are ready to introduce the coarse Baum–Connes assembly map for a metric space with bounded geometry.

\begin{Def}\label{Rips complex}
Let $X$ be a discrete metric space with bounded geometry. For each $d\geq 0$, the Rips complex of $X$ at scale $d$, denoted by $P_d(X)$, consists as a set of all formal sums
$$z=\sum_{x\in X}t_xx$$
such that each $t_x$ is in $[0,1]$, such that $\sum_{x\in X}t_x=1$, and such that the support of $z$ defined by $supp(z):=\{x\in X\,|\,t_x\ne0\}$ has diameter at most $d$.
\end{Def}

We will next discuss the semi-simplicial metric of the Rips complex of $X$. One can find more details of Rips complex in the section 7.2 of \cite{HIT2020}. We shall first define the spherical metric $d_S$ on $P_d(X)$. On each path connected component of $P_d(X)$, the spherical metric is the maximal metric whose restriction to each simplex $\{\sum_{i=0}^nt_ix_i\mid t_i\geq0,\sum_it_i=1\}$ is the metric obtained by identifying the simplex with $S^n_+$ via the map
$$\sum_{i=0}^nt_ix_i\to\left(\frac{t_0}{\sqrt{\sum_{i=0}^nt_i^2}},\cdots,\frac{t_n}{\sqrt{\sum_{i=0}^nt_i^2}}\right),$$
where $S^n_+=\{(x_0,\cdots,x_n)\in\IR^{n+1}\mid x_i\geq0,\sum_{i=0}^nx_i^2=1\}$ endowed with the standard Riemannian metric on the unit $n$-sphere.

For any $x,y\in P_d(Z)$, a semi-simplicial path $\delta$ between $x$ and $y$ is a finite sequence of points
$$x=x_0,y_0,x_1,y_1,\cdots,x_n,y_n=y$$
where $x_i$ and $y_i$ are in $X$ for each $i=\{1,\cdots,n\}$. The length of $\delta$ is defined to be
$$l(\delta)=\sum_{i=0}^{n}d(x_i,y_i)+\sum_{i=0}^{n-1}d_S(y_i,x_{i+1})$$
We define the semi-spherical metric $d_{P_d}$ on $P_d(X)$ by
$$d_{P_d}(x,y)=\inf\{l(\delta)\mid\text{ $\delta$ is a semi-simplicial path between $x$ and $y$}\}.$$
One can check that $(P_0(X),d_{P_0})$ identifies isometrically with $(X,d)$. Moreover, it has been proved in \cite[Proposition 7.2.11]{HIT2020} that the canonical inclusion $i_d:X\to P_d(X)$ is a coarse equivalence for each $d\geq 0$.

If $d<d'$, then $P_d(X)$ is included in $P_{d'}(X)$ as a subcomplex via a simplicial map. Passing to inductive
limit, we obtain the assembly map
$$\mu:\lim_{d\to\infty}K_*(P_d(X))\to\lim_{d\to\infty}K_*(C^*(P_d(X)))\cong K_*(C^*(X)).$$

\begin{CNcon}
If $X$ is a discrete metric space with bounded geometry, then the coarse assembly map
$$\mu:\lim_{d\to\infty}K_*(P_d(X))\to K_*(C^*(X))$$
is injective.
\end{CNcon}

To state our main theorem, we will also need some preparations of coarse geometry and group actions. Recall that a metric space $(X,d)$ is said to admits a coarse embedding into a Hilbert space if there exists a map $f:X\to\H$ and two non-decreasing unbounded functions $\rho_-,\rho_+:\IR_+\to\IR_+$ such that
$$\rho_-(d_X(x,x'))\leq\|f(x)-f(x')\|_{\H}\leq \rho_+(d_X(x,x')).$$
The concept of fibred coarse embedding is introduced by X.~Chen, Q.~Wang and G.~Yu in \cite{CWY2013} as a generalization of coarse embedding. In this paper, we will only focus on the case when a space admits a fibred coarse embedding into a real $\ell^p$-space as following:

\begin{Def}[\cite{CWY2013}]\label{FCE}
Let $p\geq1$ and let $B$ denote the real Banach space $\ell^p(\IN,\IR)$. A metric space $(X,d)$ is said to admit a fibred coarse embedding into $B$ if there exists
\begin{itemize}
\item a field of Banach space $(B_x)_{x\in X}$ over $X$ such that each $B_x$ is isometric to $B$;
\item a section $s:X\to\bigsqcup_{x\in X}B_x$, i.e. $s(x)\in B_x$ for each $x\in X$;
\item two non-decreasing functions $\rho_-$ and $\rho_+$ from $\IR_+$ to $\IR_+$ with $\lim_{r\to\infty}\rho_{\pm}(r)=\infty$
\end{itemize}
such that, for any $R>0$, there exists a bounded subset $K\subset X$ for which there exists a trivialization
$$t_{x,R}:(B_z)_{z\in B(x,R)}\to B(x,R)\times B$$
for each $x\in X\backslash K$, that is, a map from $(B_z)_{z\in B(x,R)}$ to the constant field $B(x,R)\times B$ such that $t_{x,R}$ restrict to the fibre $V_z$ is an affine-isometry $t_{x,R}(z):V_z\to V$, satisfying the following conditions:
\begin{itemize}
\item [(1)]for any $z_1,z_2\in B(x,R)$,
$$\rho_-(d(z_1,z_2))\leq \|t_{x,R}(z_1)(s(z_1))-t_{x,R}(z_2)(s(z_2))\|_p\leq \rho_+(d(z_1,z_2));$$
\item [(2)]for any $x,y\subset X\backslash K$ with $B(x,R)\cap B(y,R)\ne \emptyset$, there exists an affine-isometry $t_{xy,R}:V\to V$ such that $t_{x,R}(z)\circ t^{-1}_{y,R}(z)=t_{xy,R}$ for all $z\in B(x,R)\cap B(y,R)$.\qed
\end{itemize}\end{Def}

Notice that for $x\in X$ and $R'>R>0$, if there exists two trivilizations $t_{x,R}$ and $t_{x,R'}$ for $B(x,R)$ and $B(x,R')$, respectively. Without loss of generality, we may assume that
$$t_{x,R'}|_{B(x,R)}=t_{x,R}.$$
For the notational convenience, we denote $t_x$ and $t_{xy}$ for all $t_{x,R}$ and $t_{xy,R}$, respectively.

In the rest of this paper, we will prove the following theorem:

\begin{Thm}\label{main theorem}
Let $X$ be a discrete metric space with bounded geometry. If $X$ admits a fibred coarse embedding into an $\ell^p$ space, then the coarse Novikov conjecture holds for $X$.
\end{Thm}

\subsection{Warped Cones and proper affine isometric group actions}

In this subsection, we shall recall some basic notions of warped cones and proper affine group actions (cf. \cite{Roe2005,Saw2018}). These are prominent examples of spaces which admit a fibred coarse embedding into an $\ell^p$-space.

Let $(Y,d_Y)$ be a compact metric space. The open cone of $Y$, denoted by $\O Y$, is the set $Y\times [1,\infty)$ with the metric $d_{\O Y}$ defined by
$$d_{\O Y}((y_1,t_1),(y_2,t_2))=|t_1-t_2|+\min\{t_1,t_2\}\cdot\frac{d_Y(y_1,y_2)}{\diam(Y)}.$$

Let $(X,d_X)$ be a proper metric space, $\Gamma$ a finite generated countable group acting by homeomorphisms on $X$ and $S\subseteq\Gamma$ a finite generating set. The warped metric $\delta_{\Gamma}$ on $X$ is defined to be the greatest metric satisfying that
$$\delta_{\Gamma}(x,x')\leq d_X(x,x')\quad\text{and}\quad \delta_{\Gamma}(sx,x)\leq 1$$
for any $x,x'\in X$ and $s\in S$. The warped metric on $X$ always exists and the coarse structure induced by the warped metric does not depend on the choice of generating set $S$ for $\Gamma$ (see \cite[Proposition 1.7]{Roe2005}). For any $\gamma\in\Gamma$, let $|\gamma|$ denote the word length of $\gamma$ relative to the generating set $S$. Let $x,x'\in X$, the warped distance between from $x$ to $y$ is \cite{Roe2005}
$$\delta_{\Gamma}(x,x')=\inf\left\{\sum_{i=0}^{N-1}d(\gamma_ix_i,x_{i+1})+|\gamma_i|\,\Big|\, x=x_0,x_1,\cdots,x_N=x'\in X, \gamma_i\in\Gamma\right\}.$$

\begin{Def}[\cite{Saw2018}]
Let $(Y,d_Y)$ be a compact metric space and let $\Gamma$ be a finitely generated group acting on $Y$ by homeomorphisms. The \emph{warped cone} of $Y$, denoted by $\O_{\Gamma}(Y)$, is the open cone $\O Y$ with the warped metric, where the warping group action is defined by $\gamma(y,t)=(\gamma y,t)$.
\end{Def}

\begin{Def}[P.~C.~Baayen and J.~De.~Groot \cite{BaayenGroot1968}]
A group action $\Gamma\curvearrowright Y$ is said to admit a \emph{linearization} in Banach space $B$ if and there exists an isometric representation of $\Gamma$ on $B$ and $Y$ admits a bi-Lipschtz equivariant embedding into $B$ 
\end{Def}

One is referred to the section 3.1 in \cite{SW2021} for more informations about linearization.

For a residually finite group, it is well-known that the coarse geometric properties of the box sapces $\Box(\Gamma)$ are closely related to the analytic properties of $\Gamma$. A summary of the relationship can be found in \cite{CWW2013,DGWZ2022}. Actually, similar results also appear in the relationship between the coarse geometric properties $\O_{\Gamma}Y$ and the dynamical and analytic properties of $\Gamma$:

\begin{Thm}Let $\Gamma$ be a finte gnerated group and $(Y,d)$ a compact metric apace. Assume that there is a free group action $\Gamma\curvearrowright Y$.
\begin{enumerate}
\item[(1)]The warped cone $\O_{\Gamma}Y$ has property A if and only if the action $\Gamma\curvearrowright Y$ is amenable.
\item[(2)]Assume moreover the action $\Gamma\curvearrowright Y$ admit a linearization in the real Banach space $\ell^p(\IN,\IR)$ for some $p\geq 1$. If $\Gamma$ admits a proper affine isometric action on $\ell^p(\IN,\IR)$, then $\O_{\Gamma}Y$ admits fibred coarse embedding into $\ell^p(\IN,\IR)$.
\end{enumerate}\end{Thm}

The first term of the theorem above is proved by combining \cite[Theorem 3.4]{Roe2005} and \cite[Theorem 4.2]{SW2021}. The second term is proved in \cite[Theorem 3.2]{SW2021} for the case of $p=2$. Actually their prove also holds for any $p\geq 1$. One can also find another proof in \cite[Theorem 3.2]{WW2017} for the case of $p\ne 2$.

In the last of this section, we introduce a typical example constructed by profinite completion:

\begin{Exa}
Let $\Gamma$ be a residually finite group and $\{\Gamma_n\}_{n=1}^{\infty}$ a decreasing sequence of finite index normal subgroups with $\bigcap_{n=1}^{\infty}\Gamma_n=\{e\}$. We denote the quotient group by $G_n=\Gamma/\Gamma_n$. The identity map of $\Gamma$ induces an group homomorphism $G_n\to G_{n-1}$ for each $n\in\IN$. Then we have the following sequence:
$$\cdots\to G_n\to\cdots\to G_2\to G_1\to 1.$$
We denote $\wh{\Gamma}((\Gamma_n))$ the inverse limit of this sequence and call it the profinite completion of $\Gamma$ with respect to $(\Gamma_n)$. It is well-known that $\wh\Gamma((\Gamma_n))$ is a compact metrizable group containing $\Gamma$ as a dense subgroup and $\Gamma$ acts freely on $\wh\Gamma((\Gamma_n))$ by left multiplication and admit a linearization in $\ell^p(\IN,\IR)$ for any $1\leq p<\infty$ (\cite[Lemma 3.19]{SW2021}).

If $\Gamma$ admits a proper affine isometric action on $\ell^p(\IN,\IR)$ for some $p\geq 1$, then the warped cone $\O_{\Gamma}(\wh\Gamma((\Gamma_n)))$ admits a fibred coarse embedding into $\ell^p(\IN,\IR)$.
\end{Exa}

Combining with Theorem \ref{main theorem}, we have the following corollaries:

\begin{Cor}
Let $\Gamma$ be a countable discrete group, $Y$ be a compact metric space with a free $\Gamma$-action. If this action admits a linearization in the $\ell^p(\IN,\IR)$ and $\Gamma$ admits a proper affine isometric action on $\ell^p(\IN,\IR)$ for some $p\geq 1$, then the coarse Novikov conjecture holds for $\O_{\Gamma}Y$.
\end{Cor}

G.~Yu proved hyperbolic groups always admits a proper affine isometric action on $\ell^p(\IN,\IR)$ for a sufficiently large $p\geq 1$ in \cite[Theorem 1.1]{Yu2005}. In particular, we have the following:

\begin{Cor}
Let $\Gamma$ be a finitely generated residually finite hyperbolic group, and $G=\widehat{\Gamma} ((\Gamma_n))$ the corresponding profinite completion. Then the coarse Novikov conjecture holds for $\mathcal{O}_\Gamma (G)$.
\end{Cor}

\section{A Bott periodicity theorem for $\ell^p$ spaces}

In this section, we will introduce a $C^*$-algebra $\A(B)$ associated to a given real Banach space $B=\ell^p(\IN,\IR)$ and calculate its $K$-theory.

In \cite{KY2012}, G.~Kasparov and G.~Yu introduced a $C^*$-algebra for Banach space with Property (H). They pointed out that its $K$-theory can be calculated if the base space is an $\ell^p$-space. We provide a detailed proof in this section. We would like to mention that our construction is slightly different from theirs. The $C^*$-algebra $\A(X)$ in \cite{KY2012} is generated by some certain functions on a dense subspace $V$ of $B$. To build the twisted algebras in Section 6, for any affine isometry in Definition \ref{FCE}, we shall need a $C^*$-isomorphism associated to the affine isometry to translate the support of a function in $\A(X)$. However, the affine isometries in Definition \ref{FCE} may not preserve the dense subspace in \cite{KY2012}. Our construction $\A(B)$ will circumvent this problem.

\subsection{A Clifford algebras for a Hilbert space}

In this subsection, we will recall the definition of Clifford algebras for Hilbert spaces.

Let $\H$ be an infinite dimensional separable real Hilbert space. For each $n\geq 0$, denote
$$\H^{\odot n}_{\IC}=\IC\odot\underbrace{\H\odot\H\odot\cdots\odot\H}_{\text{n times}},$$
where $\odot$ means the algebraic tensor product over $\IR$ and $\H^{\otimes 0}_{\IC}=\IC$. Then $\H^{\odot n}_{\IC}$ is equipped with a natural structure of complex vector space. We define the complex inner product on $\H_{\IC}^{\odot n}$ by the formula:
$$\langle z_1\otimes v_1\otimes\cdots\otimes v_n,z_2\otimes w_1\otimes\cdots\otimes w_n\rangle=z_1\cdot\overline{z_2}\cdot\langle v_1,w_1\rangle\cdot\hdots\cdot\langle v_n,w_n\rangle,$$
and denote $\H_{\IC}^{\otimes n}$ the completion of $\H^{\odot n}_{\IC}$ under the inner product about. We denote $1\otimes v_1\otimes\cdots\otimes v_n$ simply by $v_1\otimes\cdots\otimes v_n$.

Denote $K^n$ the closed subspace of $\H^{\otimes n}_{\IC}$ which is spanned by
$$\left\{v_1\otimes\cdots\otimes v_k-\text{sgn}(\sigma)v_{\sigma(1)}\otimes\cdots\otimes v_{\sigma(n)}\,\big|\, v_1,\cdots,v_n\in\H,\text{$\sigma$ is a $k$-permutation}\right\}.$$
Define the $k$-th complex exterior power of $\H$ to be the quotient Hilbert space
$$\wedge^n\H_{\IC}=\H^{\otimes n}_{\IC}/K^n,$$
where the equivalence class in $\wedge^n\H_{\IC}$ is always denoted by $v_1\wedge\cdots\wedge v_n$. Let $\{e_1,\cdots,e_n,\cdots\}$ be an orthonormal basis of $\H$. Then the set $\{e_{i_1}\wedge e_{i_2}\wedge\cdots\wedge e_{i_n}\mid i_1<i_2<\cdots<i_n\}$ gives an orthonormal basis of $\wedge^k\H_{\IC}$. Define the antisymmetric Fock space:
$$\wedge^*\H_{\IC}=\bigoplus_{k\in\IN}\wedge^k\H_{\IC}.$$
For any $v\in\H$, we define the creation operator $C(v)$ by
$$C(v)(v_1\wedge\cdots\wedge v_k)=v\wedge v_1\wedge\cdots\wedge v_k.$$
It is not hard to check that the adjoint operator defined by
$$C^*(v)(v_1\wedge\cdots\wedge v_k)=\sum_{i=1}^k(-1)^{i+1}\langle v,v_i\rangle v_1\wedge\cdots\wedge\wh{v_i}\wedge\cdots\wedge v_k.$$
We define a self-adjoint operator
$$\hat v:=C(v)+C^*(v)\in\B(\wedge^*\H_{\IC})$$
for each $v\in\H$, then we have that
$$\hat w\hat v+\hat v\hat w=2\langle w,v\rangle.$$
The complex Cliiford algebra $\Cl(\H)$ to be subalgebra of $\B(\wedge^*\H_{\IC})$ generated by $\{\hat v\mid v\in\H\}$.

If $V$ is a linear subspace of $\H$, then $\wedge^*V_{\IC}$ is clearly a subspace of $\wedge^*\H_{\IC}$. Thus $\Cl(V)$ can be naturally viewed as a subalgebra of $\Cl(\H)$.

\begin{Pro}
Let $\{V_n\}_{n\in\IN}$ be an increasing sequence of finite-dimensional Hilbert space such that $\bigcup_{n\in\IN}V_n$ is dense in $\H$. Then
$$\lim_{n\to\infty}\Cl(V_n)\cong\Cl(\H).$$
\end{Pro}

\begin{proof}
For each $n\in\IN$, we can view $\Cl(V_n)$ as a subalgebra of $\Cl(\H)$. It suffices to prove
$$\overline{\bigcup_{n\in\IN}\Cl(V_n)}=\Cl(\H),$$
which follows directly from the fact that $\|\hat v\|_{\Cl(\H)}=\|v\|_{\H}$ and $\overline{\bigcup_{n\in\IN}V_n}=\H$.
\end{proof}

\begin{Rem}
Actually, there is another way to construct the Clifford algebra for $\H$. Let $\mathcal T(\H)=\oplus_{n=0}^{\infty}\H^{\odot n}_{\IC}$ be tha complex algebraic tensor algebra of $\H$. Let $I$ be the ideal of $\mathcal T(\H)$ which is generated by $\left\{v\otimes v-\|v\|_{\H}^2\cdot 1\mid v\in\H\right\}$. Denote $\IC l(\H)$ the quotient algebra $\mathcal T(\H)/I$. As a linear space, $\mathcal T(\H)$ is also a dense subspace of the Hilbert space $\oplus_{n=0}^{\infty}\H^{\otimes n}_{\IC}$. Denote by $H(\IC l(\H))$ the completion of $\IC l(\H)$ under the inner product induced by $\oplus_{n=0}^{\infty}\H^{\otimes n}_{\IC}$. Then $\IC l(\H)$ has a canonical faithful representation by the left multiplication on $H(\IC l(\H))$. Then the Clifford algebra
$$\Cl(\H)=\overline{\IC l(\H)^{\|\cdot\|}}$$
where the norm $\|\cdot\|$ is given by the canonical representation on $H(\IC l(\H))$.
\end{Rem}

\subsection{A $C^*$-algebra associated with a real $\ell^p$ space}

Fix $1\leq p<\infty$ and denote $B=\ell^p(\IN,\IR)$ to be the real Banach space. Specially, we denote $\H=\ell^2(\IN,\IR)$ the real Hilbert space. We define the sign function $\text{sgn}:\IR\to\IR$ by the formula
$$\text{sgn}(a)=\left\{\begin{aligned}&0&&,a=0;\\&a|a|^{-1}&&,a\ne0.\end{aligned}\right.$$
Denote by $S(B)$ the closed unit sphere of $B$. Recall the Mazur map $\psi:S(B)\to S(\H)$, defined by
$$\psi(a_1,\cdots,a_n,\cdots)=(sgn(a_1)|a_1|^{p/2},\cdots,\text{sgn}(a_n)|a_n|^{p/2},\cdots),$$
is a uniform homeomorphism, i.e., the $r$-oscillation of $\psi$ and $\psi^{-1}$ tend to $0$ as $r$ tends to $0$, where the $r$-oscillation of $\psi$ is defined to be
$$\omega_r(\psi)=\sup_{x,x'\in S(B),\|x-x'\|_p\leq r}\|\psi(x)-\psi(x')\|.$$
We consider the \emph{$p/2$-H\"older extension} of the Mazur map $\psi$, which is a map $\phi:B\to\H$ defined by 
$$\phi(x)=\left\{\begin{aligned}&0&&,x=0;\\&\|x\|^{p/2}_p\psi\left(\|x\|_p^{-1}x\right)&&,x\ne0.\end{aligned}\right.$$
Actually, we have the following uniform homeomorphism extension theorem due to E.~Odell and T.~Schlumprecht \cite[Proposition 2.9]{OS1994} and Q.~Cheng \cite[Theorem 2.4]{Cheng2016}:

\begin{Pro}[\cite{OS1994,Cheng2016}]\label{extend Mazur uniform homeo}
For any $R>0$, the extended Mazur map restricted on the ball $\phi|_{Ball_B(0,R)}:Ball_B(0,R)\to Ball_{\H}(0,R^{p/2})$ is a uniform homeomorphism, i.e.,
$$\omega_r\left(\phi|_{Ball_B(0,R)}\right)\to 0\text{ as }r\to 0.$$
\end{Pro}

As a corollary, $\phi$ forms a homeomorphism between $B$ and $\H$. When $p=2$, the extended Mazur map is the identity map on $\H$.

Denote by $\IR_+=[0,\infty)$. Let $C(B\times\IR_+,\Cl(\H\oplus\IR))$ be the algebra of all continuous functions from $B\times\IR_+$ to $\Cl(\H\oplus\IR)$. We view $\Cl(\H)$ as a subalgebra of $\Cl(\H\oplus\IR)$ in the following paper.

\begin{Def}
Define $C_{b,0}(B\times\IR_+,\Cl(\H\oplus\IR))$ to be the subalgebra of $C(B\times\IR_+,\Cl(\H\oplus\IR))$ consisting of all bounded functions $f$ such that
$$f(x,0)\in\Cl(\H)\subseteq\Cl(\H\oplus\IR)$$
for all $x\in B$.
\end{Def}

\begin{Def}
For any $x_0\in B$, we define the Clifford operator $C_{x_0}\in C(B\times\IR_+,\Cl(\H\oplus\IR))$ by
$$C_{x_0}(x,t)=(\phi(x-x_0),t)\in\H\oplus\IR\subseteq\Cl(\H\oplus\IR)$$
for any $x\in B$ and $t\in\IR_+$.
\end{Def}

We denote $C_0(\IR)_{ev}$ (or $C_0(\IR)_{odd}$, respectively) the subset of $C_0(\IR)$ of all even (or odd, respectively) functions. For any Hilbert space $H$, we next define the functional calculus of $\Cl(H)$: for any $f\in C_0(\IR)_{ev}$ and $v\in\H\subset\Cl(H)$, we define $f(v)=f(\|v\|)\in\IC\subset\Cl(H)$ and for $g\in C_0(\IR)_{odd}$, we define the functional calculus by
$$g(v)=\left\{\begin{aligned}&0&&,v=0\\&g(\|v\|)\frac{v}{\|v\|}&&,v\ne0,\end{aligned}\right.$$
where $g(v)$ is an element of $\Cl(H)$.

\begin{Lem}
For any $x_0\in B$ and $f\in C_0(\IR)$, the Bott map $\beta_{x_0}$ is defined by functional calculus such that
$$(\beta_{x_0}(f))(x,t)=f(C_{x_0}(x,t)),$$
for each $x\in B$ and $t\in\IR_+$. Then $\beta_{x_0}$ gives a graded homomorphism from $C_0(\IR)$ to $C_{b,0}(B\times\IR_+,\Cl(\H\oplus\IR))$.
\end{Lem}

\begin{proof}
Combining the definition of the functional calculus and the fact that $\phi$ is continuous, it is obvious that $\beta_{x_0}(f)$ is continuous.
\end{proof}

\begin{Def}
The algebra $\A(B)$ is the $C^*$-subalgebra $C_{b,0}(B\times\IR_+,\Cl(\H\oplus\IR))$ generated by
$$\{\beta_x(f)\mid x\in B,f\in C_0(\IR)\}.$$
Let $W$ be a subspace of $B$, the algebra $\A(B,W)$ is defined to be the $C^*$-subalgebra of $\A(B)$ generated by
$$\{\beta_x(f)\mid x\in W\subseteq B,f\in C_0(\IR)\}.$$
\end{Def}

\begin{Lem}\label{cont. Bott}
Fix $f\in\S$.\begin{enumerate}
\item[(1)]For any $\varepsilon>0$, there exists $\delta>0$ such that
$$\|\beta_x(f)-\beta_y(f)\|\leq\varepsilon$$
whenever $x,y\in B$ satisfies $\|x-y\|_p\leq\delta$.
\item[(2)]As a corollary, let $\{x_n\}\in B$ be a sequence which converges to $x_0\in B$. Then
$$\lim_{n\to\infty}\|\beta_{x_n}(f)-\beta_{x_0}(f)\|=0$$
for each $f\in C_0(\IR)$.
\end{enumerate}\end{Lem}

\begin{proof} 
Let $f_0$ and $f_1$ be the even part and odd part of $f$. There exists $R>0$ such that $|f(t)|\leq\varepsilon/2$ for all $|t|>R$. 

Case 1. For the even part, by definition, for each $x,y,z\in B$, we have that
$$\|(\beta_x(f_0)-\beta_y(f_0))(z,t)\|=|f_0(\|C_x(z,t)\|)-f_0(\|C_y(z,t)\|)|.$$ If $(z,t)$ satisfies that both $\|C_x(z,t)\|$ and $\|C_y(z,t)\|$ are greater than $R$, then
$$|f_0(\|C_x(z,t)\|)-f_0(\|C_y(z,t)\|)|\leq |f_0(\|C_x(z,t)\|)|+|f_0(\|C_y(z,t)\|)|\leq\varepsilon.$$
Notice that $\|C_x(z,t)\|=(\|\phi(z-x)\|^2+t^2)^{1/2}=(\|z-x\|^p_p+t^2)^{1/2}$.
If one of $\|C_x(z,t)\|$ or $\|C_y(z,t)\|$ is less than $R$, without loss of generality, assume that $\|C_x(z,t)\|<R$. Then $\|z-x\|^p_p< R^2$. Notice that
\begin{equation*}\begin{split}
\left|\|C_x(z,t)\|-\|C_y(z,t)\|\right|&=\left(\|z-x\|^p_p+t^2\right)^{1/2}-\left(\|z-y\|^p_p+t^2\right)^{1/2}\\
&\leq \left(\|z-x\|^p_p-\|z-y\|^p_p\right)^{1/2},
\end{split}\end{equation*}
and
$$\left|\|z-x\|_p-\|z-y\|_p\right|\leq \|x-y\|_p.$$
Set $\delta_0>0$ satisfying that $|f_0(t_1)-f_0(t_2)|<\frac\varepsilon2$ for all $t_1,t_2$ with $|t_1-t_2|<\delta_0$. 
As $y(t)=t^p$ is uniform continuous in $[0,R^{p/2}+1]$, then there exists $\delta_1>0$ such that $|t_1^p-t_2^p|<\delta_0^2$ for all $t_1,t_2\in[0,R^{p/2}+1]$ with $|t_1-t_2|<\delta_1$. Set $\delta_{ev}=\min\{\delta_1,1\}$. If $\|x-y\|_p\leq \delta_{ev}$, then both $\|z-x\|_p$ and $\|z-y\|_p$ are in $[0,R^{p/2}+1]$ and we have that
$$|f_0(\|C_x(z,t)\|)-f_0(\|C_y(z,t)\|)|\leq\frac\varepsilon2.$$

Case 2. For the odd part, assume that there exists $g\in C_0(\IR)_{ev}$ such that $f_1(t)=tg(t)$. Such elements are dense in $C_0(\IR)_{odd}$. Then, by definition,
\begin{equation*}\begin{split}
\|(\beta_x(f_1)-\beta_y(f_1))(z,t)\|=&\left\|g(\|C_x(z,t)\|)C_x(z,t)-g(\|C_y(z,t)\|)C_y(z,t)\right\|\\
\leq&\left\|g(\|C_x(z,t)\|)-g(\|C_y(z,t)\|)C_x(z,t)\right\|\\
&+|g(C_y(z,t))|\cdot\|C_x(z,t)-C_y(z,t)\|.
\end{split}\end{equation*}
Similarly, if both $\|C_x(z,t)\|$ and $\|C_y(z,t)\|$ are greater than $R$, then
$$\left\|g(\|C_x(z,t)\|)C_x(z,t)-g(\|C_y(z,t)\|)C_y(z,t)\right\|\leq |f_1(\|C_x(z,t)\|)|+|f_1(\|C_y(z,t)\|)|\leq\varepsilon.$$
If one of $\|C_x(z,t)\|$ or $\|C_y(z,t)\|$ is less than $R$, with a similar argument as above, we can find a $\delta_3>0$ such that $|g(\|C_x(z,t)\|)-g(\|C_y(z,t)\|)|\leq \frac \varepsilon{4R}$. For the second term, we have that
$$\|C_x(z,t)-C_y(z,t)\|=\|\phi(z-x)-\phi(z-y)\|.$$
By Proposition \ref{extend Mazur uniform homeo}, there exists $0<\delta_4<1$ associated with $\phi|_{Ball_B(0,R^{2/p}+1)}$ such that if $\|x-y\|\leq \delta_4$, then $\|\phi(z-x)-\phi(z-y)\|\leq\frac\varepsilon{4\|g\|}$. Set $\delta_{odd}=\min\{\delta_3,\delta_4\}$, we conclude that if $\|x-y\|_p\leq\delta_{odd}$, then
$$\|(\beta_x(f_1)-\beta_y(f_1))(z,t)\|\leq\frac\varepsilon2$$
for all $(z,t)\in B\times\IR$. 

Set $\delta=\min\{\delta_{ev},\delta_{odd}\}$, then
$$\|\beta_x(f)-\beta_y(f)\|\leq\|\beta_x(f_0)-\beta_y(f_0)\|+\|\beta_x(f_1)-\beta_y(f_1)\|\leq\varepsilon.$$
This completes the proof.
\end{proof}

\begin{Cor}
Let $\{W_n\}_{n\in\IN}$ be an increasing sequence of finite dimensional linear subspaces of $B$ with $\bigcup_{n\in\IN}W_n$ dense in $B$. Then\begin{enumerate}
\item[(1)]$\A(B,B)=\A(B)$ and $\A(B,W_n)\subseteq\A(B,W_{n+1})$;
\item[(2)]$\lim_{n\to\infty}\A(B,W_n)=\A(B,\bigcup_{n\in\IN}W_n)=\A(B,\overline{\bigcup_{n\in\IN}W_n})$.
\end{enumerate}\end{Cor}

\begin{proof}
The first term and the first equality in (2) are immediate from the definition. The second equality in (2) comes from Lemma \ref{cont. Bott}.
\end{proof}

\subsection{The $K$-theory of $\A(B)$}

In this subsection, we shall calculate the $K$-theory of $\A(B)$. We will show that the Bott map $(\beta_{x_0})_*:K_*(\S)\to K_*(\A(B))$ is an isomorphism for any $x_0\in B$. One can find a proof in \cite[Remark 7.7]{GWY2018} for the case when $p=2$. We remark that the Bott periodicity theorem holds for any Banach space which is spherical equivalent to $\ell^2$ by using a similar argument.

Let $N_n=\{0,1,\cdots,n\}\subseteq \IN$. We denote $B_n=\ell^p(N_n,\IR)$ and $H_n=\ell^2(N_n,\IR)$ to be the finite dimensional subspaces of $B$ and $\H$, respectively. It is clear that $\overline{\bigcup_{n=0}^{\infty}B_n}=B$ and $\overline{\bigcup_{n=0}^{\infty}H_n}=\H$. Denote $\S=C_0(\IR)$ equipped with the grading according to even and odd functions.

Define $C_{b,0}(B_n\times\IR_+,\Cl(H_n\oplus\IR))$ to be the subalgebra of $C_b(B_n\times\IR_+,\Cl(H_n\oplus\IR))$ consisting of all bounded function $f$ such that
$$f(x,0)\in \Cl(H_n)\subseteq\Cl(H_n\oplus\IR).$$
Since the Mazur map restricts to a uniform homeomorphism $\psi_n=\psi|_{B_n}:S(B_n)\to S(H_n)$, for any $x_0\in B_n$, we can still define the Clifford operator $C^n_{x_0}\in C(B_n\times\IR_+,\Cl(H_n\oplus\IR))$ on $B_n$ by
$$C^n_{x_0}(x,t)=(\phi_n(x-x_0),t)\in\Cl(H_n\oplus\IR),$$
where $\phi_n:B_n\to H_n$ is the extended Mazur map of $\psi_n$ and $x\in B_n$, $t\in\IR_+$. The Bott map $\beta_{x_0}^n:\S\to C_{b,0}(B_n\times\IR_+,\Cl(H_n\oplus\IR))$ is defined to be
$$\left(\beta^n_{x_0}(f)\right)(x,t)=f(C^n_{x_0}(x,t)).$$

\begin{Def}
Define $\A(B_n)$ to be the subalgebra of $C_{b,0}(B_n\times\IR_+,\Cl(H_n\oplus\IR))$ generated by
$$\{\beta^n_x(f)\mid f\in\S, x\in B_n\}.$$
\end{Def}

Similarly, one can also define the Clifford operator on $H_n$ and $\A(H_n)$ by taking $p=2$. To clearfy the notation, we write the Bott map for $H_n$ associated to $v\in H_n$ by $\beta^{H_n}_v$.

\begin{Lem}
$\A(B_n)$ is a $(B_n\times\IR_+)$-$C^*$-algebra.
\end{Lem}

\begin{proof}
Let $\A_{ev}(B_n)$ be the $C^*$-subalgebra of $\A(B_n)$ generated by
$$\left\{\beta^n_x(f)\mid x\in B_n,\ f\in C_0(\IR)_{ev}\right\}.$$
Then $\A_{ev}(B_n)$ is a subalgebra of the center $\mathcal Z(\A(B_n))$. It is clear that $\beta^n_x(f)\in C_0(B_n\times\IR_+)$ for any $x\in B_n$ and $f\in C_0(\IR)_{ev}$. For any different points $(x_1,t_1),(x_2,t_2)\in B_n\times\IR_+$ with $t_1\leq t_2$, we have that
$$\|C_{x_1}(x_1,t_1)\|=|t_1|<\sqrt{\|x_1-x_2\|_p^p+t_2^2}=\|C_{x_1}(x_2,t_2)\|.$$
Choose a suitable $f\in C_0(\IR)_{ev}$ such that $f(\|C_{x_1}(x_1,t_1)\|)\ne f(\|C_{x_1}(x_2,t_2)\|)$, i.e.,
$$\left(\beta^n_{x_1}(f)\right)(x_1,t_1)\ne\left(\beta^n_{x_1}(f)\right)(x_2,t_2)$$
By Stone-Weierstrass theorem, we conclude that $\A_{ev}(B_n)\cong C_0(B_n\times\IR_+)$.

For any $f\in C_c(\IR)$, let $f_0\in C_0(\IR)$ be an even function such that $f_0(t)=1$ for all $t\in\supp(f)$, thus $\beta^n_x(f)\beta^n_x(f_0)=\beta^n_x(f)$ for any $x\in  B_n$. This means that $C_0(B_n\times\IR_+)\cdot\A(B_n)$ is dense in $\A(B_n)$. Then $\A(B_n)$ is a $(B_n\times\IR_+)$-$C^*$-algebra.
\end{proof}

In \cite{KY2012}, G.~Kasparov and G.~Yu introduced a $C^*$-algebra $\S\C(B_n)$ associated with $B_n$. We briefly recall the definition here. Denote $\C(B_n)=C_0(B_n,\Cl(H_n))$ equipped with the grading induced from $\Cl(H_n)$. Denote $\S\C(H_n)=\S\wox\C(H_n)$ to be the graded tensor product.

\begin{Thm}\label{finite dimensional case}
With notations as above, $\A(B_n)$ is isomorphic to $\S\C(B_n)$ for any $n\in\IN$.
\end{Thm}

\begin{proof}
By definition, $\S\C(B_n)=\{f\wox h\mid f\in C_0(\IR), h\in C_0(B_n,\Cl(H_n))\}$. Define
$$\B=\{f\in C_0(\IR_+,\Cl(\IR))\mid f(0)\in\IC\}.$$
$\B$ is induced with the grading of $\Cl(\IR)\cong\IC\oplus\IC$, the scalar part and the vector part. We define a homomorphism $\varphi:\S\to\B$ by the following formula:
$$(\varphi(f))(t)=f_{ev}(t)1+f_{odd}(t)e,$$
where $f_{ev}$ and $f_{odd}$ are respectively the even and odd parts of $f$ and $e$ is a unit vector of $\IR$. It is not hard to see that $\varphi$ is an isomorphism preserving the grading.

As a corollary, we have the identification:
$$\S\C(B_n)=\{f\in C_0(B_n\times\IR_+,\Cl(H_n\oplus\IR))\mid f(x,0)\in\Cl(V_n)\},$$
where we view $\Cl(H_n)$ as a subalgebra of $\Cl(H_n\oplus\IR)$. Notice that $C_0(B_n\times\IR_+)\cong C_0(\IR)_{ev}\ox C_0(B_n)$ is a subalgebra of $\mathcal Z(\S\C(B_n))$, the center of $\S\C(B_n)$, which makes $\S\C(B_n)$ a $(B_n\times\IR_+)$-$C^*$-algebra.

Notice that $\A(B_n)$ is a subalgebra of $\S\C(B_n)$. Indeed,
$$\lim_{(x,t)\to\infty}(\beta^n_{x_0}(f))(x,t)=0\quad\text{and}\quad f(C(v,0))=f(v)\in\Cl(V).$$
It suffices to prove $\A(B_n)$ is dense in $\S\C(B_n)$. As both algebras are $(B_n\times\IR_+)$-$C^*$-algebras, it suffices to prove the fiber of $\A(B_n)$ at $(x,t)$ is the same as $\S\C(B_n)$ for any $(x,t)\in B_n\times\IR_+$. The fiber of $\A(B_n)$ at $(x,t)$ is given by
$$\A(B_n)/C_0(B_n\times\IR\backslash\{(x,t)\})\cdot\A(B_n)\cong\Cl(H_n\oplus t\IR),$$
where
$$t\IR=\left\{\begin{aligned}&0&&,t=0;\\&\IR&&,t\ne0,\end{aligned}\right.$$
which is actually same as $\S\C(B_n)$. We then finish the proof that $\A(B_n)\cong\S\C(B_n)$.
\end{proof}

\begin{Cor}The Bott map induces an isomorphism on $K$-theory:
$$(\beta_x^n)_*:K_*(\S)\xrightarrow{\cong} K_*(\A(B_n))$$
for any $x\in B_n$.
\end{Cor}

\begin{proof}
For any $x_0,x_1\in B_n$, we take $x_s=(1-s)x_0+sx_1$ for each $s\in[0,1]$. Then $s\mapsto x_s$ is a continuous path connecting $x_0$ and $x_1$. By Lemma \ref{cont. Bott}, we have that $\beta_{x_s}$ forms a homotopy between $\beta_{x_0}$ and $\beta_{x_1}$. Thus it suffices to prove the theorem for $0\in B_n$.

Define $\phi^*:\A(H_n)\to\A(B_n)$ by
$$(\phi^*(\sigma))(x,t)=\sigma(\phi(x),t),\quad\text{$\sigma\in\A(H_n)$ and $x\in B_n$}.$$
As $\phi$ is a homeomorphism, it is clearly to see that $\phi^*$ a $*$-isomorphism. For any $f\in\S$ and $(x,t)\in B_n\times\IR_+$, one can check that
\begin{equation*}\begin{split}
\left(\phi^*\left(\beta_0^{H_n}(f)\right)\right)(x,t)&=\left(\beta_0^{H_n}(f)\right)(\phi(x),t)\\
&=f(\phi(x),t)\\
&=\left(\beta_0^{n}(f)\right)(x,t).
\end{split}\end{equation*}
This shows that $\phi^*\circ\beta^{H_n}_0(f)=\beta^n_0(f)$. Thus, we have the following commuting diagram:
$$\begin{tikzcd}
K_*(\S) \arrow[rd, "(\beta^{B_n}_0)_*"'] \arrow[r, "(\beta^{H_n}_0)_*"] & K_*(\A(H_n)) \arrow[d, "\phi^*"] \\
                                                                       & K_*(\A(B_n)).                  
\end{tikzcd}$$
It is proved that $\left(\beta^{H_n}_0\right)_*$ is an isomorphism in \cite{HKT1998} and \cite{GWY2018}. As a result, $(\beta^{n}_0)_*$ is an isomorphism.
\end{proof}

For any $n\in\IN$, we denote $B'_n=\ell^p(\IN\backslash N_n,\IR)$. Then one can easily see that $B=B_n\oplus_pB'_n$, where $\oplus_p$ means that $B_n\oplus B'_n$ is equipped with the norm
$$\|(x,y)\|_p=(\|x\|^p_p+\|y\|^p_p)^{1/p}$$
for any $(x,y)\in B_n\oplus_pB'_n$. 
Denote by $P_n:B\to B_n$ and $Q_n:B\to B'_n$ the direct sum projection, i.e., for any $x=(x_1,x_2)\in B=B_n\oplus_pB'_n$ we have that $P_n(x)=x_1$ and $Q_n(x)=x_2$. The following fact which will be very useful in the next lemma can be checked with a easy calculation:

\textbf{Fact.} For any $(x,y)\in B_n\oplus_p B'_n=B$, we have that
$$\phi(x,y)=(\phi(x),\phi(y))\in H_n\oplus H_n'=\H,$$
i.e., $\phi=\phi|_{B_n}\oplus\phi_{B'_n}:B_n\oplus_pB_n'\to H_n\oplus H_n^{\bot}$.

For any $\sigma\in\A(B,B_n)$, we define the restriction homomorphism $\pi_n:\A(B,B_n)\to\A(B_n)$ by
$$(\pi_n(\sigma))(x,t)=\sigma(x,t)$$
for all $x\in W_n$. To see $\pi_n$ is well-defined, one can check that
$$(\pi_n\circ\beta_{x_0}(f))(x,t)=\pi_n(f(C_{x_0}(x,t)))=f(C^n_{x_0}(x,t)),$$
for any $x_0,x\in B_n$ and $t\geq 0$.

\begin{Lem}
The restriction homomorphism $\pi_n:\A(B,W_n)\to\A(B)$ is an isomorphism.
\end{Lem}

\begin{proof}
For each $(w,t)\in B_n'\oplus\IR$, we define
$$\tau_{w,t}^n:H_n\oplus\IR\to\H\oplus\IR$$
to be the inclusion which is identity on $H_n$ and maps $(0,(\|w\|_p^2+t^2)^{1/2})\in V_n\oplus\IR$ to $(0,\phi(w),t)$, i.e., $\tau^n_{w,t}(0\oplus\IR)$ is the one-dimensional subspace of $H_n^{\bot}\oplus\IR$ spanned by $(\phi(w),t)$ and the inclusion is an isometry. This map clearly induces an inclusion of Clifford algebras $\Cl(V_n\oplus\IR)\to\Cl(\H\oplus\IR)$.

Define
$$\Phi:C(B_n\times\IR_+,\Cl(H_n\oplus\IR))\to C(B\times\IR_+,\Cl(\H\oplus\IR))$$
by the following formula:
$$(\Phi(\sigma))(x,t)=\tau^n_{Q_n(x),t}\left(\sigma(P_n(x),(\|Q_n(x)\|_p^2+t^2)^{1/2})\right)\in\Cl(\H\oplus\IR),$$
where $(P_n(x),(\|Q_n(x)\|_p^2+t^2)^{1/2})\in B_n\times\IR_+$. It is clear that $\Phi$ is a $*$-homomorphism.

Let $x_0\in B_n$, consider the Clifford operator $C^n_{x_0}\in C(B_n\times\IR_+,\Cl(H_n\oplus\IR))$ on $B_n$, i.e., $C_{x_0}^n(x,t)=(x-x_0,t)$ for each $(x,t)\in B_n\oplus\IR$. Then for any $y\in B$, we have
\begin{equation}\begin{split}\label{eq5}
(\Phi(C^n_{x_0}))(y,t)&=\tau^n_{Q_n(y),t}(C^n_{x_0}(P_n(y),(\|Q_n(y)\|_p^2+t^2)^{1/2}))\\
&=\tau^n_{Q_n(y),t}(\phi(P_n(y)-x_0),(\|Q_n(y)\|_p^2+t^2)^{1/2})\\
&=(\phi(P_n(y)-x_0)\oplus \phi(Q_n(y)),t).
\end{split}\end{equation}
As $\phi=\phi|_{B_n}\oplus\phi|_{B_n'}:B_n\oplus_pB_n'\to H_n\oplus H_n^{\bot}$, then $\phi(x)=\phi(x_1)+\phi(x_2)$ for any $x=(x_1,x_2)\in B_n\oplus B'_n=B$. Thus $\phi(P_n(y)-x_0)+\phi(Q_n(y))=\phi(P_n(y)+Q_n(y)-x_0)=\phi(y-x_0)$. Combining the equation \eqref{eq5}, we have that $\Phi(C_{x_0}^n)=C_{x_0}$ the Clifford operator of $B$. Write $t_y=\sqrt{\|Q_n(y)\|_p^2+t^2}$. Then for any $f\in C_0(\IR)$, we have that
\begin{equation}\begin{split}\label{eq6}
(\Phi(\beta_{x_0}^{n}(f)))(y,t)=&\tau_{Q_n(y),t}^n(f(C_{x_0}^n(P_n(y),t_y)))\\
=&\tau_{Q_n(y),t}^n\left(f_{ev}\left(\|C_{x_0}^n\left(P_n(y),t_y\right)\|\right)\right)+\\
&\tau_{Q_n(y),t}^n\left(f_{odd}\left(\|C_{x_0}^n\left(P_n(y),t_y\right)\|\right)\frac{C_{x_0}^{n}(P_n(y),t_y)}{\|C_{x_0}^{n}(P_n(y),t_y)\|}\right)\\
=&f_{ev}\left(\|C_{x_0}\left(y,t\right)\|\right)+f_{odd}\left(\|C_{x_0}\left(y,t\right)\|\right)\frac{\tau_{Q_n(y),t}^n(C_{x_0}^n(P_n(y),t_y))}{\|C_{x_0}(y,t)\|}\\
=&f_{ev}\left(\|C_{x_0}\left(y,t\right)\|\right)+f_{odd}\left(\|C_{x_0}\left(y,t\right)\|\right)\frac{C_{x_0}(y,t)}{\|C_{x_0}(y,t)\|}=\beta_{x_0}(f)(y,t)
\end{split}\end{equation}
This shows that $\Phi$ maps $\A(W_n)$ to $\A(B)$ and $\Phi(\A(W_n))$ is actually equal to $\A(B,W_n)$. This also shows that
$\Phi$ is the inverse of $\pi_n$.
\end{proof}

\begin{Thm}\label{K of A(B)}
For any $x_0\in B$, the Bott homomorphism 
$$\beta_{x_0}:C_0(\IR)\to\A(B)$$
induces an isomorphism on $K$-theory, i.e.,
$$(\beta_{x_0})_*:K_*(\S)\to K_*(\A(B))$$
is an isomorphism.
\end{Thm}

\begin{proof}
Similarly, it suffices to prove the theorem for some fixed $x_0\in B$.

Viewing $B_n$ as a subspace of $B_{n+1}$, one can similarly define a map
$$\Phi_n:\A(B_n)\to\A(B_{n+1},B_n)\subseteq\A(B_{n+1}).$$
By compute on the generators, it is not hard to check that we have the following commuting diagram
$$\xymatrix{\S\ar[r]^{\beta^n_{x_0}}\ar[dr]_{\beta^{n+1}_{x_0}}&\A(B_n)\ar[d]_{\Phi_n}\ar[r]^{\Phi}&\A(B,B_n)\ar[d]\\
&\A(B_{n+1})\ar[r]^{\Phi}&\A(B,B_{n+1})}$$
where $\beta$ is defined as in \cite{HKT1998}. Combining the results we prove before, we have that
$$\A(B)=\lim_{n\to\infty}\A(B,B_n)\cong\lim_{n\to\infty}\A(B_n).$$
As the $K$-functor commutes with the direct limits (see \cite{HIT2020}), the theorem holds directly from Theorem \ref{finite dimensional case}.
\end{proof}

\section{$K$-homology at infinity and Roe algebra at infinity}

In this section, we will introduce $K$-homology at infinity and Roe algebra at infinity for a proper metric space.

The approach to $K$-homology via localization algebras is first introduced by G. Yu in \cite{Yu1997}. In this paper, we will also need to work with a version of localization algebras introduced by R. Willett and G. Yu in \cite{HIT2020}. For the convenience of the readers, we shall briefly recall its definition and its functoriality. Let $X$ be a proper metric space.

\begin{Def}[\cite{HIT2020}]\label{localization}
Let $H_X^+$ be an $X^+$-module, where $X^+$ is the one point compactification. Define $\IL[H_X^+;\infty]$ to be the collection of all bounded functions $(T_t)$ from $[1,\infty)$ to $\B(H_X^+)$ such that:
\begin{itemize}
\item[(1)]for any compact subset $K$ of $X$, there exists $t_K\geq 0$ such that for all $t\geq t_K$, the operators $\chi_KT_t$ and $T_t\chi_K$ are compact and the functions
$$t\mapsto \chi_KT_t\quad\mbox{and}\quad t\mapsto T_t\chi_K$$
are uniformly continuous when restricted to $[t_K,\infty)$;
\item[(2)]for any open neighbourhood $U$ of diagonal in $X^+\times X^+$, there exists $t_U\geq1$ such that for all $t>t_U$
$$\supp(T_t)\subseteq U.$$
\end{itemize}
Define $L^*(H_X^+;\infty)$, or simply $L^*(X^+;\infty)$, to be the $C^*$-algebra completion of $\IL[H_X^+;\infty]$ for the norm
$$\|(T_t)\|=\sup_{t}\|T_t\|_{\B(H_X)}.$$
\end{Def}

By \cite[Proposition 6.6.2]{HIT2020}, \cite[Theorem 3.2]{Yu1997} and \cite[Theorem 3.4]{QR2010}, we have the following result:

\begin{Thm}
The $K$-homology group $K_*(X)$ is isomorphic to the $K$-theory group $K_*(L^*(H_X^+;\infty))$, i.e.,
$$K_*(X)\cong K_*(L^*(H_X^+;\infty))$$
and $K_*(L^*(H_X^+;\infty))$ does not depend on the choice of ample modules up to the canonical equivalence.
\end{Thm}

Let $H_X^+$ be a geometric module. Let $\IL_0[H_X^+;\infty]$ be the collection of all $(T_t)\in\IL[H_X^+;\infty]$ such that for any compact subset $K\subseteq X$ there exists $t_K\geq 0$ such that for all $t\geq t_K$,
$$\chi_KT_t=T_t\chi_K=0.$$
It is not difficult to see that $\IL_0[H_X^+;\infty]$ is a $*$-ideal in $\IL[H_X^+;\infty]$. Let $L^*_0(H_X^+;\infty)$ be the closure of $\IL_0[H_X^+;\infty]$ inside $L^*(H_X^+;\infty)$, let
$$L^*_Q(H_X^+;\infty)=L^*(H_X^+;\infty)/L^*_0(H_X^+;\infty)$$
be the corresponding quotient $C^*$-algebra. Using an Eilenberg swindle, we have the following result as in \cite[Lemma 6.4.11]{HIT2020}

\begin{Lem}
The quotient map $L^*(H_X^+;\infty)\to L^*_Q(H_X^+;\infty)$ induces an isomorphism on $K$-theory.\hfill$\square$
\end{Lem}

To see the functoriality, we still need some preparations:

\begin{Def}
Let $K$ be a closed subspace of $X$, and let $H_X^+$ be an $X^+$-module. Define $\IL[H^+_X;K^+]$ to be the subset of $\IL[H_X^+;\infty]$ consisting of $(T_t)$ such that for any open subset $U$ of $X^+\times X^+$ that contains $K^+\times K^+$, there exists $t_U$ such that for all $t\geq t_U$
$$\supp(T_t)\subseteq U$$
Define $L^*(H^+_X;K^+)$, or simply $L^*(X^+;K^+)$, to be the closure of $\IL[H^+_X;K^+]$ inside $\IL[H_X^+;\infty]$.
\end{Def}

Similarly, we shall define $L^*_0(H_X^+;K^+)=L^*(H_X^+;K^+)\cap L^*_0(H_X^+;\infty)$ and
$$L^*_Q(H_X^+;K^+)=L^*(H_X^+;K^+)/L^*_0(H^+_X;K^+).$$
It is easy to show that the quotient map $L^*(H_X^+;K^+)\to L^*_Q(H_X^+;K^+)$ induces an isomorphism on the $K$-theory level. The following lemma has been proved in \cite[Lemma 6.3.6]{HIT2020}.

\begin{Lem}
The inclusion homomorphism from $L^*(H_K^+;\infty)$ to $L^*(H_X^+;K^+)$ induces an isomorphism from $K_*(L^*(H_K^+;\infty))$ to $K_*(L^*(H_X^+;K^+))$, i.e.,
$$K_*(K)\cong K_*(L^*(X^+;K^+))\cong K_*(L^*_Q(X^+;K^+)).$$\qed
\end{Lem}

Let $K\subset X$ be a close subset and $U=X\backslash K$, then we have the short sequence of topological spaces
$$K^+\stackrel{i}{\longrightarrow} X^+\stackrel{c}{\longrightarrow} U^+$$
where $i:K^+\to X^+$ is the inclusion and $c:X^+\to U^+$ is the collapse map which is the identity on $U$ and sends $X^+\backslash U$ to the point at infinity in $U^+$.

\begin{Pro}
With the notation above, there is a natural sic-term exact sequence
$$\xymatrix{K_0(K)\ar[r]^{i_*}&K_0(X)\ar[r]^{j^*}&K_0(U)\ar[d]\\K_1(U)\ar[u]&K_1(X)\ar[l]^{j^*}&K_1(K)\ar[l]^{i_*}}$$
induced by the short exact sequence of $C^*$-algebra
$$0\to L^*_Q(X^+;K^+)\stackrel{ad(V^i_t)}{\longrightarrow} L^*_Q(X^+;\infty)\stackrel{ad(V^c_t)}{\longrightarrow} L^*_Q(U^+;\infty)\to 0.$$
where $V^i_t$ and $V^c_t$ are the covering isometries for $i$ and $c$, respectively.
\end{Pro}

\begin{proof}
Let $d$ be a bounded metric function on $X^+$. Let $Z\subset X$ be a countable dense subset of $X$ whose intersection with $K$, $U$ is also dense in $K$, $U$ respectively for each $n\in\IN$. Let $H$ be a separable, infinite-dimensional Hilbert space, and define ample $X^+$, $K^+$ and $U^+$-module respectively by
$$H^+_X=\ell^2(Z,H)\oplus H,\, H^+_K=\ell^2(Z\cap K,H)\oplus H,\, H^+_U=\ell^2(Z\cap U,H)\oplus H$$

Define $v:\ell^2(Z\cap K,H)\to\ell^2(Z,H)$ to be the canonical inclusion. Let $u:\ell^2(Z\cap K,H)\oplus H\to H$ be an unitary operator and define
$$w=id_{\ell^2(Z\cap U,H)}\oplus u:H_X^+=\ell^2(Z\cap U,H)\oplus(\ell^2(Z\cap K,H)\oplus H)\to H_{U}^+$$
Define $V^i_t=v\oplus id_H$ and $V^c_t=w$. It is easy to see $V_t^i$ and $V_t^c$ are the covering isometries for $i$ and $c$, respectively.

Notice that $L^*_Q(X^+,K^+)$ is a two-side ideal of $L^*_Q(X^+,\infty)$ and $ad(V_t^i)$ is actually the inclusion. Thus it suffices to prove
$$\frac{L^*_Q(X^+,\infty)}{L^*_Q(X^+,K^+)}\cong L^*_Q(U^+,\infty).$$

Let $(C_n)$ be an increasing sequence of compact subsets of $U$ whose union is all of $U$. Let $\chi:\IR_+\to B(U)$ be a continuous map from $\IR_+$ to the set of all bounded Borel functions on $U$ defined by
$$\chi(t)=(n+1-t)\chi_{C_n}+(t-n)\chi_{C_{n+1}}$$
if $t\in[n,n+1)$, where $\chi_{C_n}$ is the characteristic function of $C_n$. Provisionally define maps
$$\Phi:\frac{L^*_Q(X^+,\infty)}{L^*_Q(X^+,K^+)}\to L^*_Q(U^+,\infty),\quad [[T_t]]\mapsto [ad(V^c_t)(T_t)]$$
and
$$\Psi:L^*_Q(U^+,\infty)\to \frac{L^*_Q(X^+,\infty)}{L^*_Q(X^+,K^+)},\quad [T_t]\mapsto [[\chi(t)T_t\chi(t)]],$$
where $\chi(t)T_t\chi(t)$ can be view as a operator on $\B(H_X^+)$ by viewing $\ell^2(Z\cap U,H)$ as a subspace of $H^+_X$ for each $t\in\IR_+$. It is easy to see $ad(V_t^c)(T_t)$ and $\chi(t)T_t\chi(t)$ satisfy the conditions in Definition \ref{localization}.

First, we show that $\Phi$ and $\Psi$ are well-defined $*$-homomorphisms. Let $[T_t]\in L^*_Q(X^+,K^+)$. For any compact subset $C\subset U$, we have that $d(C,K^+)>0$. Take $\varepsilon<d(C,K^+)$, there exists $T>0$ such that
$$\supp(T_t)\subset\{(x,y)\in X\times X\mid d(x,K^+)<\varepsilon,d(y,K^+)<\varepsilon\},$$
for all $t>T$. Then $\chi_C(ad(V^c_t)(T_t))=\chi_CT_t=0$ for all $t>T$, which means that $ad(V_t)(T_t)\in L^*_0(U^+,\infty)$. Therefore, $\Phi$ is a well-defined map $*$-homomorphism.

Similarly, one can show $\Psi$ is also well-defined. Indeed, let $T_t\in L^*_0(U^+,\infty)$. For any $\varepsilon>0$, denote $\N_{\varepsilon}(K^+)=\{x\in X^+\mid d(x,K^+)<\varepsilon\}$. Then $C=X^+\backslash\N_{\varepsilon}(K^+)\subset U$ is a compact subset of $U$. By definition, there exists $T>0$ such that $\chi_CT_t=T_t\chi_C=0$ for all $t>T$. Thus, for all $t>T$, we have that
$$\supp(\chi(t)T_t\chi(t))\subset \N_{\varepsilon}(K^+)\times\N_{\varepsilon}(K^+),$$
i.e., $\chi(t)T_t\chi(t)\in L^*_Q(X^+,K^+)$. By \cite[Lemma 6.1.2]{HIT2020}, the multiplier of $L^*_Q(U^+,\infty)$ defined by $\chi(t)$ is central. Thus $\Psi$ is a well-defined $*$-homomorphism.

At last, we show that $\Psi$ induces a mutually inverse of $\Phi$. Computing, for $[T_t]\in L^*_Q(U^+,\infty)$
$$\Phi(\Psi([T_t]))=[ad(V_t^c)(\chi(t)T_t\chi(t))]=[\chi(t)T_t\chi(t)],$$
here we view $\chi(t)T_t\chi(t)$ as an operator on $H^+_U$ for each $t\in\IR_+$. Notice that for any compact subset $C\subset U$,
$$\chi_C(T_t-\chi(t)T_t\chi(t))=\chi_CT_t-\chi(t)\chi_CT_t\chi(t).$$
As $\lim_{t\to\infty}\prop(T_t)=0$ and $U=\bigcup_{n\in\IN}C_n$, there exists $T>0$ and a compact subset $C'\subset U$ such that $\chi_C\chi(t)=\chi_C$ and $\chi_CT_t=\chi_CT_t\chi_{C'}$ for all $t>T$. Thus $\chi_C(T_t-\chi(t)T_t\chi(t))=0$ for all $t>T$, i.e., $[T_t]=[\chi(t)T_t\chi(t)]$.

On the other hand, for $[[T_t]]\in\frac{L^*_Q(X^+,\infty)}{L^*_Q(X^+,K^+)}$,
$$\Psi(\Phi([[T_t]]))=[[\chi(t)ad(V^c_t)(T_t)\chi(t)]]=[\chi(t)T_t\chi(t)],$$
here we view $\chi(t)T_t\chi(t)$ as an operator on $H^+_X$ for each $t\in\IR_+$. For any $\varepsilon>0$, let $C=X^+\backslash\N_{\varepsilon}(K^+)\subset U$ be a compact subset of $U$. Then by a similar argument above
$$\chi_C(T_t-\chi(t)T_t\chi(t))=0$$
for sufficiently large $t$, this completes the proof.
\end{proof}

Let $X$ be a unbounded proper metric space. Fix a based point $x_0\in X$, and let $K_R=\{x\in X\mid d(x,x_0)\leq R\}$. Denoted by $U_R=X\backslash K_R$. Thus we have the following commuting diagram by \cite[Proposition B.2.3]{HIT2020}:
$$\xymatrix{K_0(K_R)\ar[r]^{i_*}&K_0(X)\ar[r]^{j^*}&K_0(U_R)\ar[d]\\K_1(U_R)\ar[u]&K_1(X)\ar[l]^{j^*}&K_1(K_R)\ar[l]^{i_*}}$$
where $i_*$ is induced by the natural inclusion $i:K_R\to X$ and $j^*$ is induced by the collapse $K_R$ to $\infty$ map $c:X^+\to U_R^+$ .

If $R<R'\in\IR$, we denote $i_{RR'}:K_R\to K_{R'}$ and $j_{RR'}:U_{R'}\to U_R$ the inclusion map. Let $c_{RR'}:U^+_R\to U^+_{R'}$ be the collapse map that is identity on $U$ and sends $U^+\backslash U_{R'}$ to the point at infinity in $U^+$. Then we have the following commuting diagram
\begin{equation}\label{CD1}\xymatrix{K_R\sqcup\{\infty\}\ar[r]^{i}\ar[d]^{i_{RR'}}&X^+\ar[r]^{c}\ar@{=}[d]&U_R^+\ar[d]^{c_{RR'}}\\
K_{R'}\sqcup\{\infty\}\ar[r]^{i}\ar[d]&X^+\ar[r]^{c}\ar@{=}[d]&U_R^+\ar[d]\\\vdots&\vdots&\vdots}\end{equation}
Taking the direct limit as $R$ tends to infinity, we get a commutative diagram
$$\xymatrix{\lim\limits_{R\to\infty}K_0(K_R)\ar[r]^{i_*}&K_0(X)\ar[r]^{j^*}&\lim\limits_{R\to\infty}K_0(U_R)\ar[d]\\\lim\limits_{R\to\infty}K_1(U_R)\ar[u]&K_1(X)\ar[l]^{j^*}&\lim\limits_{R\to\infty}K_1(K_R)\ar[l]^{i_*}}$$

\begin{Def}
The $K$-homology at infinity of $X$, denote by $K_*^{\infty}(X)$, is defined to be the group $\lim\limits_{R\to\infty}K_*(U_R)$.
\end{Def}

For any locally compact, secondly countable, Hausdorff space $X$, we can define a unbounded proper metric on $X$. We can still define the $K$-homology at infinity for $X$ according to this metric. By the definition, we can see that $K^{\infty}_*(X)$ does not depend on the choice of the metric. Moreover, if $X$ is compact, then $K^{\infty}_*(X)$ is zero group. Recall the definition of the representable $K$-homology:

\begin{Def}[\cite{HIT2020}]
Let $H_X$ be an $X$ module. Define $R\IL[H_X]$ to be the collection of all elements $(T_t)$ of $\IL[H^+_X,\infty]$ such that there exists a compact subset $K$ of $X$ and $t_K\geq 1$ such that
$$T_t=\chi_KT_t\chi_K$$
for all $t\geq t_K$. Define $RL^*(H_X)$ to be the completion of $R\IL[H_X]$ for the norm
$$\|(T_t)\|=\sup_{t}\|T_t\|_{\B(H_X)}.$$

The representable $K$-homology of $X$, denoted by $RK_*(X)$, is defined to be the $K$-theory group
$$RK_*(X):=K_*(RL^*(H_X))$$
\end{Def}

\begin{Rem}
It is obvious that $RL^*(H_X)$ is a closed ideal of $L^*(H_X^+,\infty)$. Notice that $\lim\limits_{R\to\infty}K_*(K_R)$ is actually the representable $K$-homology \cite[Proposition 9.4.7]{HIT2020}. Thus the diagram can be rewrite as following:
$$\xymatrix{RK_0(X)\ar[r]^{i_*}&K_0(X)\ar[r]^{j^*}&K^{\infty}_0(X)\ar[d]\\K^{\infty}_1(X)\ar[u]&K_1(X)\ar[l]^{j^*}&RK_1(X)\ar[l]^{i_*}.}$$
\end{Rem}

Let $H_X$ be an ample $X$-module, $L^*(H_X^+,\infty)$ and $RL^*(H_X)$ be as above. Define $L^*_{\infty}(H_X)$ to be the quotient algebra $L^*(H_X^+,\infty)/RL^*(H_X)$.

\begin{Pro}
The $K$-homology at infinity $K_*^{\infty}(X)$ is isomorphic to the $K$-theory group of $L^*_{\infty}(H_X)$, i.e.,
$$K_*^{\infty}(X)\cong K_*(L^*_{\infty}(H_X)).$$
\end{Pro}

\begin{proof}
Let $Z\subset X$ be a countable dense subset of $X$ whose intersection with $K_n$, $U_n$ is also dense in $K_n$, $U_n$ respectively for each $n\in\IN$. Let $H$ be a separable, infinite-dimensional Hilbert space, and define ample $X^+$, $K_n^+$ and $U^+_n$-module respectively by
$$H^+_X=\ell^2(Z,H)\oplus H,\, H^+_{K_n}=\ell^2(Z\cap K_n,H)\oplus H,\, H^+_{U_n}=\ell^2(Z\cap U_n,H)\oplus H$$

For each $n,n'\in\IN$, assume that $n<n'$, define $v_n:\ell^2(Z\cap K_n,H)\to\ell^2(Z,H)$ and $v_{nn'}:\ell^2(Z\cap K_n,H)\to\ell^2(Z\cap K_{n}',H)$ to be the canonical inclusion. Let $u_n:\ell^2(Z\cap K_n,H)\oplus H\to H$ and $u_{nn'}:\ell^2(U_{n}\cap K_{n'},H)\oplus H\to H$ be unitaries and define
$$w_n=id_{\ell^2(Z\cap U_n,H)}\oplus u_n:H_X^+=\ell^2(Z\cap U_n,H)\oplus(\ell^2(Z\cap K_n,H)\oplus H)\to H_{U_n}^+$$
and
$$w_{nn'}=id_{\ell^2(Z\cap U_{n'},H)}\oplus u_{nn'}:H_{U_n}^+=\ell^2(Z\cap U_{n'},H)\oplus(\ell^2(U_{n}\cap K_{n'},H))\to H_{U_{n'}}^+$$
Define $(V_t)_n=v_n$, $(V_t)_{nn'}=v_{nn'}$, $(W_t)_n=w_n$ and $(W_t)_{nn'}=w_{nn'}$. Then we have the following commuting diagram:
$$\xymatrix{0\ar[r]&L^*_Q(X^+,K^+_n)\ar[rr]^{ad((V_t)_n)}\ar[d]^{ad((V_t)_{nn'})}&&L^*_Q(X^+,\infty)\ar[rr]^{ad((W_t)_{n})}\ar@{=}[d]&&L^*_Q(U_n^+,\infty)\ar[d]^{ad((W_t)_{nn'})}\ar[r]&0\\
0\ar[r]&L^*_Q(X^+;K^+_{n'})\ar[rr]^{ad((V_t)_{n'})}\ar[d]&&L^*_Q(X^+;\infty)\ar[rr]^{ad((W_t)_{n'})}\ar@{=}[d]&&L^*_Q(U_{n'}^+;\infty)\ar[d]\ar[r]&0\\&\vdots&&\vdots&&\vdots&}$$
Taking the direct limit as $n$ tends to infinity, by \cite[Proposition 9.4.7]{HIT2020}, we have that
$$0\to RL^*_Q(X^+;\infty)\to L^*_Q(X^+;\infty)\stackrel{J}{\longrightarrow}\lim\limits_{n\to\infty}L^*_Q(U^+_n;\infty)\to0.$$
We denote the map
$$J=\lim_{n\to\infty}ad((W_t)_n):L^*_Q(X^+;\infty)\to\lim\limits_{n\to\infty}L^*_Q(U^+_n;\infty).$$
For any $T_t\in R\IL[X^+;\infty]$, there exists a compact subset $K\subset X$ and $t_K>1$ such that $T_t=\chi_KT_t\chi_K$ for all $t\geq t_K$. Set $N\in\IN$ such that $K\subset K_N$, then $[ad((W_t)_N)(T_t)]=0\in L^*_Q(U_N^+,\infty)$. Thus $J([T_t])=0$ for all $[T_t]\in RL^*_Q(X^+,\infty)$. Then $J$ descends to a map on the quotient:
$$J_{\infty}:L^*_{\infty,Q}(X^+;\infty)\to\lim\limits_{n\to\infty}L^*_Q(U^+_n;\infty).$$
Thus we have the following commuting diagram
$$\xymatrix{0\ar[r]&RL^*_Q(X^+;\infty)\ar[rr]\ar@{=}[d]&&L^*_Q(X^+;\infty)\ar[rr]\ar@{=}[d]&&L^*_{\infty,Q}(X^+;\infty)\ar[d]^{J_{\infty}}\ar[r]&0\\
0\ar[r]&RL^*_Q(X^+;\infty)\ar[rr]&&L^*_Q(X^+;\infty)\ar[rr]&&\lim\limits_{n\to\infty}L^*_Q(U^+_n;\infty)\ar[r]&0}$$
Thus $(J_{\infty})_*:K^{\infty}_*(X)\to K_*(L^*_{\infty}(H_X))$ is an isomorphism by Five Lemma, as desire.
\end{proof}

For technical convenience, we will consider the following version of localization algebra and representable localization algebra (c.f. \cite{Yu1997,Yu2000}). Let $H_0$ be an infinite dimensional Hilbert space, and $Z\subseteq X$ be a countable dense subset and $\IC_f[X]$ be as in Definition \ref{matrix Roe algebra}.

\begin{Def}[\cite{Yu1997}]
The algebraic localization algebra, denoted by $\IC_L[X]$, is the $*$-algebra of all uniformly bounded and uniformly continuous functions
$$g:[0,\infty)\to\IC_f[X]\subseteq\IC[X]$$
such that the family $(g(t))_{t\in[0,\infty)}$ satisfy the conditions in Definition \ref{matrix Roe algebra} with uniform constants and there exists a bounded function $R(t):[0,\infty)\to[0,\infty)$ with $\lim_{t\to\infty}R(t)=0$ such that
$$(g(t))(x,y)=0\mbox{ whenever }d(x.y)>R(t)$$
for all $x,y\in Z$ and $t\in[0,\infty)$.

Define $C^*_L(X)$ to be the completion of $\IC_L[X]$ with respect to the norm
$$\|g\|=\sup_{t\in\IR_+}\|g(t)\|.$$
\end{Def}

\begin{Thm}[\cite{Yu1997,QR2010}]
Suppose that $X$ is a proper metric space with bounded geometry, then there exists a local index map $\mu_L:K_*(P_d(X))\to K_*(C^*_L(P_d(X)))$ which is an isomorphism.
\end{Thm}

Consequently, if $X$ is a discrete metric space with bounded geometry, we have the following commuting diagram:
$$\xymatrix{&&\lim\limits_{d\to\infty}K_*(P_d(X))\ar[d]^{\mu}\ar[dll]_{\mu_L}\\\lim\limits_{d\to\infty}K_*(C^*_L(P_d(X)))\ar[rr]^{ev_*}&&K_*(C^*(X)),}$$
where the evaluation homomorphism $ev:C^*_L(X)\to C^*(X)$ is defined by $ev(f(t))=f(0)$ for all $f\in C^*_L(X)$. Then, to proof the coarse Novikov conjecture for $X$, it suffices to prove the evaluation homomorphism
$$ev_*:\lim_{d\to\infty}K_*(C^*_L(P_d(X)))\to\lim_{d\to\infty}K_*(C^*(P_d(X)))$$
is an injection.

\begin{Def}
The representable localization algebra of $X$, denoted by $RC^*_L(X)$, is the closure of the collection of all elements $g(t)$ of $C^*_L(X)$ such that there exists a compact subset $K$ of $X$ such that
$$g(t)=\chi_Kg(t)\chi_K$$
for all $t\in[0,\infty)$. It is not hard to see that $RC^*_L(X)$ is a two-sided closed ideal of $C^*_L(X)$.
\end{Def}

The proof of the following lemma is similar with \cite[Proposition 9.4.2, Proposition 6.6.2]{HIT2020} and \cite[Theorem 3.2]{Yu1997}.

\begin{Lem}
Let $X$ be a proper metric space. Then we have
$$K_*(L^*(H_X^+,\infty))\cong K_*(C^*_L(X))\quad\mbox{and}\quad K_*(RL^*(H_X))\cong K_*(RC^*_L(X)).$$
As a consequence, the $K$-homoloogy at infinity $K^{\infty}_*(X)$ is isomorphic to $K_*(C^*_L(X)/RC^*_L(X))$.\hfill$\square$
\end{Lem}

Next, we will introduce the Roe algebra at infinity for a proper metric space $X$. Fix $x_0\in X$, let $K_n=B(x_0,n)$ and $U_n=X\backslash K_n$. Let $Z$ be a countable subset of $X$.

\begin{Def}
The algebraic Roe algebra at infinity, denoted by $\IC_{\infty}[X]$, is the set of all equivalence classes $[T]$, where $T\in \IC_f(X)$ and the equivalence relation $\sim$ on $[T]$ is defined by $T\sim S$ if and only if
$$\lim_{n\to\infty}\sup_{x,y\in U_n}\|T(x,y)-S(x,y)\|=0.$$

The algebraic operation of $\IC_{\infty}[X]$ is given by the usual matrix operations of $\IC_f[X]$. Define the Roe algebra at infinity $C^*_{\infty}(X)$ to be the completion of $\IC_{\infty}[X]$ with respect to the norm
$$\|[T]\|=\limsup_{n\to\infty}\|\chi_{U_n}T\|.$$
\end{Def}

One can see that the Roe algebra at infinity is also invariant up to coarse equivalence, so $C^*_{\infty}(X)$ is isomorphic to $C^*_{\infty}(P_d(X))$ for each $d\geq 0$. If $X$ is bounded, then $C^*_{\infty}(X)=0$.

\begin{Def}
The localization algebra at infinity of $X$, denoted by $\IC_{L,\infty}[X]$, is defined to be the $*$-algebra of all bounded and unformly norm-continuous functions
$$g:[0,\infty)\to\IC_{\infty}[X]$$
such that $g$ is of the form $g(t)=[T_t]$ with $T_t\in \IC[X]$ where the family $(T_t)_{t\in\IR_+}$ satisfies the conditions in Definition \ref{matrix Roe algebra} with uniform constants and there exists a bounded function $R(t):[0,\infty)\to[0,\infty)$ with $\lim_{t\to\infty}R(t)=0$ such that
$$(T_t)(x,y)=0\mbox{ whenever }d(x.y)>R(t)$$
for all $x,y\in Z$ and $t\in[0,\infty)$.

Define $C^*_{L,\infty}(X)$ to be the completion of $\IC_{L,\infty}[X]$ with respect to the norm
$$\|g\|=\sup_{t\in\IR_+}\|g(t)\|.$$
\end{Def}

There exists a canonical quoitent map $C^*_L(X)\to C^*_{L,\infty}(X)$ defined by $g\mapsto g'$ where $g'(t)=[g(t)]$ for each $t\in\IR_+$. One can also see that $g'=0$ for any $g\in RC^*_L(X)\cap \IC_L[X]$ as $g(t)$ is uniformly supported in some compact subset of $X$. Then we have a homomorphism
$$q:\frac{C^*_L(X)}{RC^*_L(X)}\to C^*_{L,\infty}(X)$$
which induces a homomorphism on $K$-theory
$$q_*:K_*^{\infty}(X)\to K_*(C^*_{L,\infty}(X)).$$

\begin{Thm}
Suppose $X$ is a discrete metric space with bounded geometry. Then
$$q_*:K_*^{\infty}(P_d(X))\to K_*(C^*_{L,\infty}(P_d(X)))$$
is an isomorphism for each $d\geq 0$.
\end{Thm}

\begin{proof}
By using an Eilenberg swindle argument as \cite[Theorem 3.4]{Yu1997}, one can show that both groups above are invariant under strongly Lipschitz homotopies and both groups admits a Mayer-Vietories sequence. Then it is suffices to prove the theorem for the case when $P_d(X)$ is $0$-dimensional, which obviously holds as both groups are trivial.
\end{proof}

\section{Reduction to the coarse Novikov conjecture at infinity}

In this section, we discuss the strategy to prove the main theorem. As general metric spaces (such as warped cones) are always not of the form of coarse disjoint unions of a sequence of bounded subsets, we can not use the techniques in \cite{CWY2013} directly. We will show that there exists an assmbly map at infinity from the $K$-homology group at infinity to the $K$-theory group of the Roe algebra at infinity, and the coarse Novikov conjecture can always be reduced to verifying the assembly map at infinity to be injective whether the space is of the form of coarse disjoint union or not.

\subsection{The case of coarse disjoint unions}

In this section, we assume that $X=\bigsqcup_{n\in\IN}X_n$ is a coarse disjoint union of a sequence of uniformly bounded geometry, finite metric spaces such that $\# X_n\to\infty$ as $n\to\infty$. Actually, this situation has been discussed in Section 4 of \cite{CWY2013}.

\begin{Lem}
Let $X=\bigsqcup_{n\in\IN}X_n$ be a coarse disjoint union of a sequence of uniformly bounded geometry, finite metric spaces such that $\# X_n\to\infty$ as $n\to\infty$. Then for sufficiently large $d$, we have that
$$K_*^{\infty}(P_d(X))=\frac{\prod_{n=0}^{\infty}K_*(P_d(X_n))}{\bigoplus_{n=0}^{\infty}K_*(P_d(X_n))}.$$
\end{Lem}

\begin{proof}
For any $d\geq 0$, there exists $N_d\in N$ large enough such that $d(X_n,X_m)> d$ for $n,m\in N_d$. Let $X_{N_d}=\bigcup^{N_d-1}_{n=0} X_n$, then we have that $P_d(X)=P_d(X_{N_d})\cup\mathop{\sqcup}\limits_{n>N_d}P_d(X_n)$. Then the maps $RK_*(P_d(X_n))\to RK_*(P_d(X))$ and $K_*(P_d(X_n))\to K_*(P_d(X))$ induced by the inclusions $X_n\to X$ induce isomorphisms
$$RK_*(P_d(X))\cong K_*(P_d(X_{N_d}))\oplus\bigoplus_{n>N_d}K_*(P_d(X_n))$$
and
$$K_*(P_d(X))\cong K_*(P_d(X_{N_d}))\oplus\prod_{n>N_d}K_*(P_d(X_n)).$$
The inclusion $\bigoplus_{n>N_d}K_*(P_d(X_n))\to\prod_{n>N_d}K_*(P_d(X_n))$ is an injection, i.e., $RK_*(P_d(X))\to K_*(P_d(X))$ is an injection. Thus, we have the short exact sequence
$$0\to RK_*(P_d(X))\to K_*(P_d(X))\to K^{\infty}_*(P_d(X))\to0$$
This means that
$$K_*^{\infty}(P_d(X))=\frac{K_*(P_d(X))}{RK_*(P_d(X))}=\frac{\prod_{n=0}^{\infty}K_*(P_d(X_n))}{\bigoplus_{n=0}^{\infty}K_*(P_d(X_n))},$$
which completes the proof.
\end{proof}

For each $d>0$. let $Z_d\subset P_d(X)$ be a countable dense subset, and let $Z_{d,n}=Z_d\cap P_d(X_n)$ for any $d\geq0$, $n\in\IN$. We assume that $Z_{d,n}\subset Z_{d,n'}$, if $d'<d$.
In this case, the Roe algebra at infinity can be also described as follow:

\begin{Def}[\cite{CWY2013}]\label{Roe algebra}
For each $d\geq0$, define $\ICu[\PdX]$ to be the set of all equivalence classes $T=[(T^{(0)},\cdots,T^{(n)},\cdots)]$ of sequences $(T^{(0)},\cdots,T^{(n)},\cdots)$ described as follows:\begin{itemize}
\item[(1)]$\{T^{(n)}\}_{n\in\IN}$ is a family of uniformly bounded functions, where $T^{(n)}$ is a function from $Z_{d,n}\times Z_{d,n}$ to $\mathcal K$ for all $n\in\IN$;
\item[(2)]for any bounded subset $B\subset P_d(X_n)$, we have that
$$\#\{(x,y)\in B\times B\cap Z_{d,n}\times Z_{d,n}\mid T^{(n)}(x, y)\ne 0\}<\infty;$$
\item[(3)]there exists $L>0$ such that
$$\#\{y\in Z_{d,n}|T^{(n)}(x,y)\ne0\}<L,\qquad\#\{y\in Z_{d,n}\mid T^{(n)}(y,x)\ne 0\}<L$$
for all $x\in Z_{d,n}$, $n\in\IN$;
\item[(4)]there exists $R>0$ such that $T^{(n)}(x,y)=0$ whenever $d(x,y)>R$ for $x,y\in Z_{d,n}$, $n\in\IN$. The least such $R$ is called the propagation of the sequence $(T^{(0)},\cdots,T^{(n)},\cdots)$.
\end{itemize}
The equivalence relation $\sim$ on these sequences is defined by
$$(T^{(0)},\cdots,T^{(n)},\cdots)\sim(S^{(0)},\cdots,S^{(n)},\cdots)$$
if and only if
$$\lim_{n\to\infty}\sup_{x,y\in Z_{d,n}}\|T^{(n)}(x,y)-S^{(n)}(x,y)\|_{\K}=0.$$
\end{Def}

By viewing $T^{(n)}$ as $Z_{d,n}\times Z_{d,n}$ matrices, the product structure for $\ICu[\PdX]$ is
defined as the usual matrix operations. Define $C^*_{u,\infty}(\PdX)$ to be the completion
of $\ICu[\PdX]$ with respect to the norm
$$\|T\|=\limsup_{n\to\infty}\|T^{(n)}\|,$$
where each operator $T^{(n)}$ is viewed as an element of the Roe algebra $C^*(P_d(X_n))$.

The following notion of localization algebra has its origin in \cite{Yu1997}, we shall recall its relation with $K$-homology at infinity.

\begin{Def}\label{localization algebra}
Let $\ICuL[\PdX]$ be the set of all bounded, uniformly norm-continuous functions
$$g:\IR_+\to \ICu[\PdX]$$
such that $g(t)$ is of the form $g(t)=[g^{(0)}(t),\cdots,g^{(n)}(t),\cdots]$ and satisfies that there exists a bounded function $R(t):\IR_+\to \IR_+$ with $\lim_{t\to\infty}R(t)=0$ such that
$(g^{(n)}(t))(x,y)=0$ whenever $d(x,y)>R(t)$ and $n\in\IN$;

The localization algebra at infinity $\CauL(\PdX)$ is defined to be the norm completion of $\ICuL[\PdX]$, where $\CauL(\PdX)$ is endowed with the norm
$$\|g\|_{\infty}=\sup_{t\in\IR_+}\|g(t)\|.$$
\end{Def}

One can see that the $\Cau(\PdX)$ coincides with $C^*_{\infty}(X)$ and $\CauL(\PdX)$ coincides with $C^*_{L,\infty}(P_d(X))$. 

For each $d>0$, there is a $*$-homomorphism
$$\Phi:\IC_f[P_d(X)]\to\ICu[\PdX]$$
defined by $\Phi(T)=[(\Phi^{(0)}(T),\cdots,\Phi^{(n)}(T),\cdots)]$ for $T\in\IC_f[P_d(X)]$, with
$$\Phi^{(n)}(T)(x,y)=\left\{\begin{aligned}&0,&&\mbox{if }n<N_R\\&T|_{Z_{d,n}\times Z_{d,n}},&&\mbox{if }n\geq N_R,\end{aligned}\right.$$
where $R = \prop(T)$ and $N_R\in\IN$ is large enough such that
$$d_{P_d(X)}\left(P_d(X_n),P_d\left(\bigsqcup_{i=0}^{n-1}X_i\right)\right)>R,$$
for all $n\geq N_R$. Then $\Phi$ extends to a $C^*$-homomorphism $\Phi:C^*(P_d(X))\to\Cau(\PdX)$, see \cite[Theorem 4.5]{CWY2013}.

For any $d\geq 0$, there exists $N_d\in N$ large enough such that $d(X_n,X_m)> d$ for $n,m\in N_d$. Let $X_{N_d}=\bigcup^{N_d-1}_{n=0} X_n$. As $P_d(X)=P_d(X_{N_d})\bigsqcup\left(\sqcup_{n\geq N_d}P_d(X_n)\right)$, we have that
$$K_*(P_d(X))=K_*(P_d(X_{N_d}))\bigoplus\prod_{n=N_d}^{\infty}K_*(P_d(X_n)).$$
By the definition of the assembly maps, we have the following commutative diagram
$$\xymatrix{0\ar[d]&&0\ar[d]\\K_*(P_d(X_{N_d}))\oplus\bigoplus_{n=N_d}^{\infty}K_*(P_d(X_n))\ar[rr]\ar[d]&&K_*(\mathcal{K})\ar[d]\\
K_*(P_d(X))\ar[rr]^{ev_*}\ar[d]&&K_*(C^*(P_d(X)))\ar[d]^{\Phi_*}\\
K^{\infty}_*(P_d(X))\ar[rr]^{ev_{\infty,*}}\ar[d]&&K_*(C^*_{u,\infty}(P_d(X_n)))\\0&&}$$
where $ev_{\infty,*}$ is induced by the evaluation map 
$$ev_{\infty}:C^*_{L,\infty}(P_d(X))\to C^*_{\infty}(X),\qquad g\mapsto g(0).$$
By \cite[Remark 2.12]{OY2009}, one has that $K_*(\K)\to K_*(C^*(P_d(X)))$ is an injection.

Passing to inductive limit as $d\to\infty$, the top horizontal arrow is an isomorphism for the following reason. An element in the sum, as a finite sequence, is supported on summands below some fixed $m$ and, as $d\to\infty$, will eventually be absorbed into the first term on a single simplex. Thus, to prove $\mu$ is injective,  it suffices to prove $\mu_{\infty}$ is injective by diagram chasing.

\subsection{The case of spaces which admit an infinite coarse component}

The next notion is introduced by M.~Finn-Sell in \cite{Martin2014}.

\begin{Def}[\cite{Martin2014}]
A bounded geometry proper metric space $X$ is said to have an infinite coarse component if there exists $R>0$ such that $P_R(X)$ has an unbounded connected component. Otherwise, we say $X$ only has finite coarse components.
\end{Def}

\begin{Lem}\label{fcc and cdu}
A bounded geometry proper metric space $X$ has only finite coarse components if and only if it is a coarse disjoint union of a sequence of finite metric spaces.
\end{Lem}

\begin{proof}
Let $X=\bigsqcup_{n\in\IN}X_n$ be a coarse disjoint union of the sequence of finite metric spaces $\{X_n\}$. For each $R>0$, there exists $N>0$ such that $d(X_i,X_j)>R$ for all $i,j>N$. Let $X_{N_R}=\bigsqcup_{n=1}^NX_n$ and we have that $P_R(X)=P_R(X_{N_R})\sqcup \bigsqcup_{n>N}P_R(X_n)$. Thus $X$ has only finite coarse components.

On the other hand, let $X$ be a metric space which has only finite coarse components. Fix $x_0\in X$, denote $P_d(X)_{x_0}$ the connected component of $P_d(X)$ which contains $x_0$. Then let
\begin{equation}\begin{split}\nonumber
X_1&=P_{2^1}(X)_{x_0}\cap X;\\
X_2&=(P_{2^2}(X)_{x_0}\cap X)\backslash X_1;\\
\cdots&\\
X_n&=(P_{2^n}(X)_{x_0}\cap X)\backslash X_{n-1};\\
\cdots&.
\end{split}\end{equation}
Then we have that $X=\bigsqcup_{n\in\IN}X_n$. As $X$ has bounded geometry and finite coarse components, $X_n$ is a finite metric space for each $n\in\IN$. Assume that $i,j\in\IN$ ($i>j$) and any two points $x_i\in X_i$, $x_j\in X_j$. As $x_i\notin P_{2^j}(X)_{x_0}$, we have that
$$d(x_i,x_j)>2^j,$$
which means that $d(X_i,X_j)\to\infty$ as $i,j\to\infty$. Thus $X$ is the coarse disjoint union of the sequence $\{X_n\}_{n\in\IN}$.
\end{proof}

\begin{Lem}\label{zero map}
Let $X$ be a bounded geometry metirc space which admits an infinite coarse component. Then
\begin{itemize}
\item[(1)]the $K$-theory map $i_*:K_*(\K)\to K_*(C^*(X))$ induced by the canonical inclusion $i:\K\to C^*(X)$ is the zero map, where $\K=\K(H_X)$ is the algebra of all compact operators on the geometric module $H_X$;
\item[(2)]the map $\lim_{d\to\infty}RK_*(P_d(X))\to\lim_{d\to\infty}K_*(P_d(X))$ induced by the canonical inclusion $RC_L^*(P_d(X))\to C^*_L(P_d(X))$ is also the zero map.
\end{itemize}\end{Lem}

\begin{proof}
As $X$ has an infinite coarse component and bounded geometry property, then there exists $R>0$ and $x_0\in X$ for which there is a unbounded component contianing $x_0\in P_R(X)$. Thus there exists a sequence $\{x_i\}_{i=1}^{\infty}\subset X$ such that $\lim_{i\to\infty}d(x_i,x_0)=\infty$ and $d(x_i,x_{i+1})\leq R$ for all $i\in\IN$. Let $K_n=\{x\in X\mid d(x_0,x)\leq n\}$ for all $n\in\IN$. As $X$ is proper metric space, for any compact subset $K\subset X$, there exists $n\in\IN$ such that $K\subset K_n$. We can identify $\K(H_X)$ by $\lim_{n\to\infty}C^*(K_n)$, where $C^*(K_n)$ is the Roe algebra of $K_n$.

To prove Property (1), we will only need to show that $(i_n)_*:K_*(C^*(K_n))\to C*(X)$ is the zero map for all $n\in\IN$, where $(i_n)_*$ is induced by the canonical inclusion $i_n:K_n\to X$. By the definition of $\{x_i\}$, we have that $K_n\cup\{x_i\}$ is coarse equivalent to $\IR_+$. Thus, $i_n$ can be seen the composition $K_n\to K_n\cup\{x_n\}\to X$. Notice that $K_*(C^*(\IR_+))=0$, then the following diagram commutes
$$\xymatrix{K_*(C^*(K_n))\ar[r]^{(i_n)_*}\ar[dr]&K_*(X)\\&K_*(C^*(K_n\cup\{x_i\}))\cong K_*(C^*(\IR_+))=0\ar[u]}.$$
Then $(i_n)_*$ is the zero map. The canonical inclusion $K_n\to K_{n+1}$ induces an isomorphism $K_*(\K)\to K_*(\K)$. Thus we have that
$$\xymatrix{K_*(C^*(K_n))\ar[r]^{(i_n)_*}\ar[d]_{=}&K_*(C^*(X))\ar[d]^{=}\\
K_*(C^*(K_{n+1}))\ar[r]^{(i_{n+1})_*}&K_*(C^*(X))}.$$
Therefore, $i_*:K_*(\K)=\lim_{n\to\infty}K_*(K_n)\to K_*(C^*(X))$ is the zero map.

We will next prove Property (2). Notice that $RK_*(P_d(X))=\lim_{n\to\infty}K_*(P_d(K_n))$. Then it suffices to show
$$\lim_{d\to\infty}\lim_{n\to\infty}K_*(P_d(K_n))\to \lim_{d\to\infty}K_*(P_d(X))$$
is the zero map for all $n\in\IN$, where the map is induced by the canonical inclusion $P_d(K_n)\to P_d(X)$ for each $n\in\IN$ and $d>0$. Notice that we have the following commutative diagram
$$\xymatrix{K_*(P_d(K_n))\ar[r]\ar[d]&K_*(P_{d'}(K_n))\ar[d]\\
K_*(P_d(K_{n'}))\ar[r]&K_*(P_{d'}(K_{n'}))}.$$
for all $d'>d$, $n'>n$, where the arrows are all given by the inclusions. Thus we have that
$$\lim_{d\to\infty}RK_*(P_d(X))\cong\lim_{d\to\infty}\lim_{n\to\infty}K_*(P_d(K_n))\cong\lim_{n\to\infty}\lim_{d\to\infty}K_*(P_d(K_n)).$$
Now it suffices to show that, for each $n\in \IN$, $K_*(P_d(K_n))\to K_*(P_d(X))$ is the zero map for sufficiently large $d$.

For given $n\in\IN$, let $R_n$ be a real number such that $\diam(K_n)<R_n$ and $d(x_i,x_{i+1})<R_n$ for all $i\geq 0$. For each $d>R_n$, the inclusion $P_d(K_n)\to P_d(X)$ can be seen as the composition
$$P_d(K_n)\to P_d(K_n\cup\{x_i\})\to P_d(X).$$
Notice that $P_d(K_n\cup\{x_i\})$ is homotopy equivalent to $\IR_+$ and $K_*(\IR_+)=0$. Thus we have the following commuting diagram
$$\xymatrix{K_*(P_d(K_n))\ar[r]\ar[dr]&K_*(P_d(X))\\&K_*(P_d(K_n\cup\{x_i\}))\cong K_*(\IR_+)=0\ar[u]}.$$
This shows that $\lim_{d\to\infty}K_*(P_d(K_n))\to\lim_{d\to\infty}K_*(X)$ is the zero map for each $n\in\IN$. Taking the direct limit as $n$ tends to infinity, this completes the proof.
\end{proof}

Let $g(t)\in RC^*_L(X)$ be a function such that there exists a compact subset $K\subset X$ such that $g(t)=\chi_Kg(t)\chi_K\in\K$ for all $t\in[0,\infty)$. Consider the evaluation map $ev:C^*_L(X)\to C^*(X)$ defined by $ev(g)=g(0)$. Notice that the restriction of $ev$ to $RC^*_L(X)$ is a $*$-homomorphism $RC^*_L(X)\to\K$.

Define $\Phi:C^*(X)\to C^*_{\infty}(X)$ by $T\mapsto [T]$ and the evaluation map at infinity
$$ev_{\infty}:C^*_{L,\infty}(P_d(X))\to C^*_{\infty}(P_d(X))\quad\text{ by }\quad[g]\mapsto [g(0)].$$
Then for each $d\geq0$, we have the following commuting diagram:
$$\xymatrix{0\ar[r]&RC^*_L(P_d(X))\ar[r]\ar[d]_{ev}&C^*_L(P_d(X))\ar[r]\ar[d]_{ev}&C^*_{L,\infty}(P_d(X))\ar[d]_{ev_{\infty}}\ar[r]&0\\
0\ar[r]&\K\ar[r]&C^*(P_d(X))\ar[r]^{\Phi}&C^*_{\infty}(P_d(X))}$$
and the commuting diagram of $K$-theory:
$$\xymatrix{\cdots\ar[r]&\lim\limits_{d\to\infty}RK_*(P_d(X))\ar[r]\ar[d]_{ev_*}&\lim\limits_{d\to\infty}K_*(P_d(X))\ar[r]\ar[d]_{ev_*}&\lim\limits_{d\to\infty}K_*^{\infty}(P_d(X))\ar[d]_{(ev_{\infty})_*}\ar[r]&\cdots\\
&K_*(\K)\ar[r]&K_*(C^*(P_d(X)))\ar[r]^{\Phi_*}&K_*(C^*_{\infty}(P_d(X)))&}.$$
Notice that the bottom sequences of both diagrams are not exact at $C^*(P_d(X))$ and $K_*(C^*(P_d(X)))$, respectively. To prove $ev_*:\lim\limits_{d\to\infty}K_*(P_d(X))\to K_*(C^*(X))$ is an injection, it suffices to prove $(ev_{\infty})_*$ is an injection by diagram chasing and Lemma \ref{zero map}.

We summarize this section by the following theorem:

\begin{Thm}\label{CNC at inf}
To prove the coarse Novikov conjecture for $X$, it suffices to prove the coarse Novikov conjecture at infinity for $X$, i.e., the $K$-theoretic homomorphism induced by the evaluation map at infinity
$$(ev_{\infty})_*:\lim_{d\to\infty}K_*(P_d(X))\to K_*(C^*_{\infty}(X))$$
is injective.
\end{Thm}

\section{Twisted algebras at infinity and their $K$-theories}

In the rest of this paper, we shall prove the evaluation map at infinity
$$(ev_{\infty})_*:\lim_{d\to\infty}K_*(C^*_{L,\infty}(P_d(X)))\to K_*(C^*_{\infty}(X))$$
is an injection, where $X$ is a bounded geometry metric space which admits a fibred coarse embedding into $\ell^p(\IN,\IR)$. The strategy of the proof is to build the following commuting diagram for each $d\geq 0$:
\begin{equation}\label{strategy}\xymatrix{K_*(C^*_{L,\infty}(P_d(X)))\ar[rr]^{(\beta_{L})_*\qquad}\ar[d]_{(ev_{\infty})_*}&&K_*(C^*_{L,\infty}(P_d(X),\A(B)))\ar[d]^{(ev^{\A}_{\infty})_*}\\K_*(C^*_{\infty}(P_d(X)))\ar[rr]^{\beta_*\qquad}&&K_*(C^*_{\infty}(P_d(X),\A(B)))}\end{equation}
where $C^*_{L,\infty}(P_d(X),\A(B))$ and $C^*_{\infty}(P_d(X),\A(B))$ are twisted localization algebra and twisted Roe algebra, respectively. We will show that the evaluation homomorphism $ev_{\infty}^{\A}$ between the twisted algebras induces an isomorphism on $K$-theory and $(\beta_L)_*$ is an injection. By a diagram chasing argument, we can see that $(ev_{\infty})_*$ is injective.

In this section, we will focus on proving $ev_{\infty}^{\A}$ in \eqref{strategy} induces an isomorphism on $K$-theory. We first introduce the twisted algebras at infinity, which are constructed by using the fact that $X$ admits a fibred coarse embedding into $\ell^p(\IN,\IR)$. The basic ideal of this section comes from \cite{Yu2000,CWY2013,CWY2015}.

\subsection{The twisted algebras at infinity}

Let $X$ be a discrete metric space with bounded geometry which admits a fibred coarse embedding into $B=\ell^p(\IN,\IR)$.

\begin{Def}\label{translation}
Let $t:B\to B$ be an affine isometry. We define a homomorphism
$$t_*:\A(B)\to\A(B)$$
by
$$t_*(\beta_x(g))=\beta_{t(x)}(g).$$
for any generator $\beta_x(g)\in\A(B)$ where $x\in B$ and $g\in\S$.
\end{Def}

For each $d\geq 0$, the Rips complex $P_d(X)$ endowed with the semi-spherical metric admits a fibred coarse embedding into $B$ as $X$ is coarse equivalent to $P_d(X)$. We give some explainations as follow. As $X$ is countable, we write $X=\{x_1,\cdots,x_n,\cdots\}$, For each $n\geq 1$, we define
$$B_{d,x_n}:=\left\{\sum_{i\in\IN}t_ix_i\in P_d(X)\,\Big|\,t_i=0\text{ for all $i<n$ and }t_n\ne0 \right\}.$$
Then it is obvious that $B_{r,x}$ is contained in the union of the simplices that contain $x$ as a vertex and $\{B_{d,x}\}_{x\in X}$ forms a disjoint Borel cover of $P_d(X)$ by definition.  For any $x\in P_d(X)$, there exists a unique $\bar x\in X$ such that $x\in B_{d,\bar x}$. Define the Borel map
$$j_d:P_d(X)\to X\quad\text{ by }\quad x\mapsto\bar x.$$
It is easy to check $j_d$ is a coarse equivalence for each $d\geq 0$. Thus let $B_x=B_{\bar x}$ and $s(x)=s(\bar x)$ for each $x\in P_d(X)$. We define $t_x(y)=t_{\bar x}(\bar y)$ for any $x,y\in P_d(X)$ with $\bar y\in B(\bar x,R)$. One can check that the field of $V^0$, sections and trivializations defined above satisfy the conditions in Definition \ref{FCE}.

Take a countable dense subset $Z_d\subset P_d(X)$ consisting of all rational points in $P_d(X)$, i.e. the point $\sum_{x\in X}t_xx\in P_d(X)$ with all coefficients $t_x$ taking rational value. Fix $x_0\in X$, denote $K_{d,m}=B_{P_d(X)}(x_0,2^m)$ to be the bounded subset of $P_d(X)$ for each $m\in\IN$ and $d\geq0$ and $U_{d,m}=X\backslash K_{d,m}$. By the definition of fibred coarse embedding, there exists a sequence of non-negative real number $(l_{d,m})_{m\in\IN}$ such that
\begin{itemize}
\item[(1)]there exists a trivialization $t_x$ for $B(x,l_{d,m})$ as in Definition \ref{FCE} for each $x\in U_{d,m}$ and $m\in\IN_+$;
\item[(2)]$(l_{d,m})$ is non-decreasing and unbounded, i.e. $0\leq l_{d,1}\leq\cdots\leq l_{d,m}\leq\cdots$ and $\lim\limits_{m\to\infty}l_{d,m}=\infty$.
\end{itemize}

Now, we are ready to define the twisted algebras.

\begin{Def}\label{twisted Roe algebra}
For each $d\geq 0$, define $\IC_{\infty}[P_d(X),\A(B)]$ to be the set of all equivalence classes of $[T]$, where $T:Z_d\times Z_d\to\A(B)\wox\K$ is a bounded function satisfying the following conditions\begin{itemize}
\item[(1)]for any bounded subset $F\subset X$, we have that
$$\#\{(x,y)\in(F\times F)\cap(Z_d\times Z_d)\mid T(x,y)\ne 0\}<\infty;$$
\item[(2)]there exists $L>0$ such that
$$\#\{y\in Z_d\mid T(x,y)\ne0\}\leq L,\quad\#\{y\in Z_d\mid T(y,x)\ne0\}\leq L$$
for all $x\in Z_d$;
\item[(3)]there exists $R\geq 0$ such that $T(x,y)=0$ whenever $d(x,y)>R$ for $x,y\in Z_d$, we denote $\prop(T)=R$ the propagation of the representative element $T$;
\item[(4)]there exists $r>0$ such that for all $x,y\in Z_d$, we have that
$$\supp(T(x,y)\subseteq B_{B\times \IR_+}(t_x(x)(s(x)),r)$$
where, $t_x=t_{x,l_{d,m}}$ is the trivialization for $x\in U_{d,m}\cap K_{d,m+1}$ and $s$ is the section as in Definition \ref{FCE}.
\end{itemize}
The equivalence relationship $\sim$ is defined by $T\sim S$ if and only if
$$\lim_{m\to\infty}\sup_{x,y\in U_{d,m}}\|T(x,y)-S(x,y)\|_{\A(B)\wox\K}=0.$$
\end{Def}

The product structure for $\IC_{\infty}[P_d(X),\A(B)]$ is defined as follow. For any $[T],[S]\in\IC_{\infty}[P_d(X),\A(B)]$, we define
$$TS:Z_d\times Z_d\to\A(B)\wox\K$$
to be the function such that there exists a sufficiently large $M\in\IN$ depending only on the $\prop(T)$ and $\prop(S)$ such that $TS(x,y)=\sum_{z\in Z_d}T(x,z)\left((t_{xz})_*(S(z,y))\right)$ for all $(x,y)\in U_{d.m}\times U_{d,m}$ with $m>M$ where $(t_{zx})_*$ is as in Definition \ref{translation} and $T(x,y)=0$ otherwise. The product of $[T]$ and $[S]$ is defined to be $[TS]$.

The $*$-structure for $\IC_{\infty}[P_d(X),\A(B)]$ is defined by the formula
$$[T]^*=[T^*]$$
where
$$(T^*)(x,y)=(t_{xy})_*(T(y,x)^*)$$
for all $x,y\in U_{d,m}$ with $m$ large enough and $0$ otherwise.

\begin{Rem}
We will give some explanations to the product structure and $*$-structure for $\IC_{\infty}[P_d(X),\A(B)]$ here. For any $[T]\in\IC_{\infty}[P_d(X),\A(B)]$, there exists $R>0$ such that $T(x,z)=0$ for all $d(x,z)>R$. Then there exists $M>0$ such that $l_{d,m}\geq 2R$ for all $m>M$. Thus $t_{xz}$ is well-defined for all $x,z\in U_{d,m}$. Combining with the condition (2) of Definition \ref{FCE}, we can check that $TS$ defined above satisfies the condition (4) of Definition \ref{twisted Roe algebra}. Moreover, one can also check that the $*$-structure is also well-defined in a similar way.
\end{Rem}

Let $\bigoplus_{x\in X}\A(B)\wox\K$ be the $C^*$-algebraic direct sum, i.e. the algebra of all bounded sequences $(a^x)_{x\in X}$ with $a^x\in\A(B)\wox\K$ for each $x$. Let
$$B_c\left(Z_d,\bigoplus_{x\in X}\A(B)\wox\K\right)$$
be the set of all bounded functions $Z_d\to\bigoplus_{x\in X}\A(B)\wox\K$ with finite supports. For notational convenince, we write the element of $B_c\left(Z_d,\bigoplus_{x\in X}\A(B)\wox\K\right)$ by $\sum_{z\in Z_d}a_z[z]$. Consider $\IE$ to be the subset of $B_c\left(Z_d,\bigoplus_{x\in X}\A(B)\wox\K\right)$ with all elements $\sum_{z\in Z_d}a_z[z]$ with compact support satisfying that
\begin{itemize}
\item[(1)]$a_z=(a_z^x)$ and $a_z^x\in\A(B)\wox\K$ for each $z\in Z_d$ and $x\in X$;
\item[(2)]$a_z^x=0$ if $d(x,z)>l_{d,m}$ for all $z\in K_{d,m+1}\cap U_{d,m}$.
\end{itemize}
Then $\IE$ is a pre-Hilbert module over $\A(B)\wox\K$:
$$\left\langle\sum_{z\in Z_d}a_z[z],\sum_{z\in Z_d}b_z[z]\right\rangle=\sum_{z\in Z_d,x\in X}(a^x_z)^*b^x_z;$$
$$\left(\sum_{z\in Z_d}a_z[z]\right)a=\sum_{z\in Z_d}a_za[z],$$
where $(a_za)^x=a_z^xa$ for any $a\in\A(B)\wox\K$. Define the Hilbert module $E$ to be the completion of $\IE$.

For any given $[T]\in\IC_{\infty}[P_d(X),\A(B)]$, let $T$ act on $E$ by
$$T\left(\sum_{z\in Z_d}a_z[z]\right)=\sum_{z\in Z_d}\left(\sum_{y\in Z_{d,n}}T(z,y)a_y\right)[z],$$
where $(T(z,y)a_y)^x=(t_{xz})_*(T(z,y))a^x_y\in\A(B)\wox\K$ and the sum is finite for the condition (2) in Definition \ref{twisted Roe algebra}.

By Definition \ref{twisted Roe algebra}, for given $[T]\in\IC_{\infty}[P_d(X),\A(B)]$, one can verify that the representation element $\chi_{U_{d,m}}T\chi_{U_{d,m}}$ is a bounded module homomorphism which has an adjoint module homomorphism for large enough $m$ and the adjoint is compatible with the $*$-structure of $T$, which gives $[T]$ a $C^*$-norm. We shall check some details in the following remark.

\begin{Rem}
Assume that $[T]\in\IC_{\infty}[P_d(X),\A(B)]$ has propagation $R$ and the representation element $T(x,y)=0$ for all $x,y\in K_{d,m}$ with $m$ satisfies that $l_{d,m}>2R$. The reason $T$ acts as a bounded module homomorphism is similar to the proof of \cite[Proposition 12.2.4]{HIT2020}. We will only show how the $*$-structure of $T$ fits with the inner product.

For given $[T]\in\IC_{\infty}[P_d(X),\A(B)]$ , there exists $R>0$ such that $T(z,y)=0$ whenever $d(y,z)>R$ for $x,z\in Z_d$. By the definition, we can comupute that
\begin{equation}\begin{split}\nonumber
\left\langle \sum_{z\in Z_d}a_z[z], T\sum_{z\in Z_d}b_z[z]\right\rangle&=\sum_{x\in X,z\in Z_d}(a^x_z)^*\left(\sum_{y\in Z_{d,n}}(T(z,y)b_y)^x\right)\\
&=\sum_{x\in X,y,z\in Z_d}(a_z^x)^*(t_{xz})_*\left(T(z,y)\right)b^x_y.
\end{split}\end{equation}
Similarly, we can also comupute that
\begin{equation}\begin{split}\nonumber
\left\langle T^*\sum_{z\in Z_d}a_z[z],\sum_{z\in Z_d}b_z[z]\right\rangle&=\sum_{x\in X,z\in Z_d}\left(\sum_{y\in Z_d}(T^*(z,y)a_y)^x\right)^*(b_z^x)\\
&=\sum_{x\in X,y,z\in Z_d}\left((t_{xz})_*(T^*(z,y))a^x_y\right)^*b_z^x.
\end{split}\end{equation}
Combining the $*$-structure of $\IC_{\infty}[P_d(X),\A(B)]$ and the fact that $t_{xy}=t_{xz}t_{zy}$ whenever $z\in B(x,R)\cap B(y,R')$ such that there exists a trivialization on $B(x,R)$ and $B(y,R')$ as in Definition \ref{FCE}, we have that
\begin{equation}\begin{split}\nonumber
\sum_{x\in X,y,z\in Z_d}\left((t_{xz})_*(T^*(z,y))a^x_y\right)^*b_z^x&=\sum_{x\in X,y,z\in Z_d}(a_y^x)^*(t_{xy})_*(T(y,z))b_z^x\\
&=\sum_{x\in X,y,z\in Z_d}(a_z^x)^*(t_{xz})_*\left(T(z,y)\right)b^x_y.
\end{split}\end{equation}
Thus we have show that the representation is compatible with the the $*$-structure. With a similar argument, we can check that the representation is also compatible under the multiplication.
\end{Rem}

\begin{Def}
The twisted Roe algebra at infinity $C^*_{\infty}(P_d(X),\A(B))$ is defined to be the operator norm closure of $\IC_{\infty}[P_d(X),\A(B)]$ with respect to the norm
$$\|[T]\|=\limsup_{m\in\IN}\|\chi_{U_{d,m}}T\|,$$
where $T$ and $\chi_{U_{d,m}}$ are viewed as bounded operators on $E$.
\end{Def}

\begin{Def}\label{twisted localization algebra}
Let $\IC_{L,\infty}[P_d(X),\A(B)]$ be the set of all bounded, uniformly norm-continuous functions
$$g:\IR_+\to\IC_{\infty}[P_d(X),\A(B)]$$
such that $g$ can be viewed as an equivalent class $[g]$ of $g$, where $g(t)$ is a function from $Z_d\times Z_d\to\A(B)\wox\K$ for each $t$ and satisfies the following conditions\begin{itemize}
\item[(1)]there exists a bounded function $R(t):\IR_+\to\IR_+$ with $\lim\limits_{t\to\infty}R(t)=0$ such that $(g(t))(x,y)=0$ whenever $d(x,y)>R(t)$;
\item[(2)]there exists $r>0$ such that
$$\supp(g(t))\subseteq B_{B}(t_x(x)(s(x)),r)$$
for all $t\in\IR_+$ and $x,y\in Z_d$, where $x\in K_{d,m+1}\cap U_{d,m}$ for some $m>0$ and $t_x=t_{x,l_{d,m}}:(B_z)_{z\in B(x,l_{d,m})}\to B(x,l_{d,m})\times B$ is the trivialization.
\end{itemize}
We remark that $[g]=[h]$ if there exists $m>0$ such that $\chi_{U_m}(g(t)-h(t))=0$ for all $t\in\IR_+$.
\end{Def}

\begin{Def}
The twisted localization algebra at infinity $C^*_{L,\infty}(P_d(X),\A(B))$ is defined to be the norm completion of $\IC_{L,\infty}[P_d(X),\A(B)]$ with respect to the norm
$$\|g\|=\sup_{t\in\IR_+}\|g(t)\|.$$
\end{Def}

The evaluation homomorphism at infinity
$$ev_{\infty}^{\A}:C^*_{L,\infty}(P_d(X),\A(B))\to C^*_{\infty}(P_d(X),\A(B))$$
is defined by $ev_{\infty}^{\A}(g)=g(0)$. which induces a homomorphism on $K$-theory:
$$(ev_{\infty}^{\A})_*:K_*(C^*_{L,\infty}(P_d(X),\A(B)))\to K_*(C^*_{\infty}(P_d(X),\A(B))).$$

\subsection{The twisted coarse Baum-Connes conjecture at infinity}

In this subsection, we will prove the following theorem

\begin{Thm}\label{twisted Baum-Connes}
The $K$-theoretic evaluation homomorphism at infinity
$$(ev_{\infty}^{\A})_*:K_*(C^*_{L,\infty}(P_d(X),\A(B)))\to K_*(C^*_{\infty}(P_d(X),\A(B)))$$
is an isomorphism.
\end{Thm}

The strategy is similar to the proof in \cite{CWY2013,Yu2000} by using cutting and pasting techniques. To begin with, we will first introduce the ideals of the twisted algebras at infinity associated with open subsets of $B$.

\begin{Def}\label{coherent}
A collection $O=(O_x)_{x\in X}$ of open subsets of $B$ is said to be coherent if there exists $r>0$ such that for all but finite $m\in\IN$, we have that
\begin{equation}\label{eq*}t_{xy}(O_y\cap B(t_y(y)(s(y)),r)=O_x\cap B(t_x(y)(s(y)),r),\end{equation}
for all $x,y\in U_{d,m}$ with $d(x,y)\leq l_{d,m}$.
\end{Def}

\begin{Exa}
Fix $r>0$. We define
$$O_x=\bigcup_{z\in B_{X_n}(x,l_n)}B(t_x(z)(s(z)),r)$$
for all $x\in X$. Then the collection
$$O=(O_x)_{x\in X}$$
is certainly a coherent system of open subsets.
\end{Exa}

For any two coherent collections $O^{(1)}$ and $O^{(2)}$ of open subsets of $V^0$, we say $O^{(1)}\subset O^{(2)}$ if $O^{(1)}_x\subset O^{(2)}_x$ for all $x\in X$. Denote $O^{(1)}\cup O^{(2)}=(O^{(1)}_x\cup O^{(2)}_x)$ and $O^{(1)}\cap O^{(2)}=(O^{(1)}_x\cap O^{(2)}_x)$. It is easy to check both $O^{(1)}\cap O^{(2)}$ and $O^{(1)}\cup O^{(2)}$ are coherent.

\begin{Def}
Let $O=(O_x)_{x\in X}$ be a coherent family of open subsets of $B$. We define
$$\IC_{\infty}[P_d(X),\A(B)]_O$$
to be the $*$-subalgebra of $\IC_{\infty}[P_d(X),\A(B)]$ generated by the equivalence class of $[T]$ such that
$$\supp(T(x,y))\subseteq O_{\bar x}$$
for all $x,y\in Z_d$ with $x\in B_{d,\bar{x}}$ and $n\in\IN$, where $\bar x$ is defined such that $x$ belongs to the Borel set $B_{d,\bar x}$ as before.

Define
$$\IC_{L,\infty}[P_d(X),\A(B)]_O$$
to be the $*$-subalgebra of $\IC_{L,\infty}[P_d(X),\A(B)]$ consisting of all functions
$$g:\IR_+\to\IC_{\infty}[P_d(X),\A(B)]_O.$$
Define $C^*_{\infty}(P_d(X),\A(B))_O$ and $C^*_{L,\infty}(P_d(X),\A(B))_O$ to be the norm closures of $\IC_{\infty}[P_d(X),\A(B)]_O$ and $\IC_{L,\infty}[P_d(X),\A(B)]_O$, respectively.
\end{Def}

It is easy to check that $C^*_{\infty}(P_d(X),\A(B))_O$ is a well-defined two-side ideal of $C^*_{\infty}(P_d(X),\A(B))$ for each coherent family of open set $O$. The equation \eqref{eq*} in Definition \ref{coherent} guarantees that the algebraic structure can be induced from $C^*_{\infty}(P_d(X),\A(B))$.

\begin{Lem}\label{MV-argument}
Let $O^{(1)}$ and $O^{(2)}$ be coherent families of open subsets of $V^0$. Then we have
\begin{itemize}
\item[(1)]$C^*_{\infty}(P_d(X),\A(B))_{O^{(1)}}+C^*_{\infty}(P_d(X),\A(B))_{O^{(2)}}=C^*_{\infty}(P_d(X),\A(B))_{O^{(1)}\cup O^{(2)}}$;
\item[(2)]$C^*_{\infty}(P_d(X),\A(B))_{O^{(1)}}\cap C^*_{\infty}(P_d(X),\A(B))_{O^{(2)}}=C^*_{\infty}(P_d(X),\A(B))_{O^{(1)}\cap O^{(2)}}$;
\item[(3)]$C^*_{\infty}(P_d(X),\A(B))_{O^{(1)}}+C^*_{\infty}(P_d(X),\A(B))_{O^{(2)}}=C^*_{\infty}(P_d(X),\A(B))_{O^{(1)}\cup O^{(2)}}$;
\item[(4)]$C^*_{\infty}(P_d(X),\A(B))_{O^{(1)}}\cap C^*_{\infty}(P_d(X),\A(B))_{O^{(2)}}=C^*_{\infty}(P_d(X),\A(B))_{O^{(1)}\cap O^{(2)}}$;
\end{itemize}\end{Lem}

\begin{proof}
The proof is similar to the proof of \cite[Lemma 6.3]{Yu2000}. We will only prove part (1) and the rest can be proved in a similar way. It suffices to show that for any $[T]\in\IC_{\infty}[P_d(X),\A(B)]_{O^{(1)}\cup O^{(2)}}$, there exists $[T^{(1)}]$ and $[T^{(2)}]$ such that $[T^{(1)}]+[T^{(2)}]=[T]$ and
$$[T^{(i)}]\in\IC_{\infty}[P_d(X),\A(B)]_{O^{(i)}},\qquad i=1,2.$$

With no loss of generality, for each $x\in Z_d$, we assume there exist bounded subsets $C^{(1)}_x$ and $C^{(2)}_x$ such that
$$\supp(T(x,y))\subseteq C^{(1)}_x\cup C^{(2)}_x\subseteq O^{(1)}_x\cup O^{(2)}_x$$
for all $y\in Z_d$. By using a partition of unity similar as \cite[Lemma 6.3]{Yu2000}, there exists $h_x^{(1)}\in C_0(O^{(1)}_x)$ and $h_x^{(2)}\in C_0(O_x^{(2)})$ such that
$$h_x^{(1)}(v)+h_x^{(2)}(v)=1$$
for all $v\in C_x^{(1)}\cup C_x^{(2)}$. Define
$$T^{(1)}(x,y)=h_x^{(1)}T(x,y)\quad\text{ and }\quad T^{(2)}(x,y)=h_x^{(2)}T(x,y).$$
Then we have
$$T^{(1)}+T^{(2)}=T\quad\text{ and }\quad [T^{(i)}]\in\IC_{\infty}[P_d(X),\A(B)]_{O^{(i)}},\qquad i=1,2.$$
We complete the proof.
\end{proof}

\begin{Pro}\label{coarse disjoint union}
Let $G$ be a subset of $X$ and $O=(O_x)_{x\in X}$ be a coherent collection of open subsets of $B$ such that \begin{itemize}
\item[(1)]each $O_x$ splits as a disjoint union $O_x=\bigsqcup_{g\in G\cap B(x,l_{d,m})}O_{x,g}$ of open subsets;
\item[(2)]there exists $r>0$ for each $x$ such that
$$O_{x,g}\subseteq B(t_x(g)(s(g)),r);$$
\item[(3)]for any $x,y\in U_{d,m}$ with $g\in B(x,l_{d,m})\cap B(y,l_{d,m})$, we have that
$$t_{xy}(O_{y,g})=O_{x,g}.$$
\end{itemize}
Then the $K$-theortic homomorphism
$$ev_*:\lim_{d\to\infty}K_*(C^*_{L,\infty}(P_d(X),\A(B))_O)\to \lim_{d\to\infty}K_*(C^*_{\infty}(P_d(X),\A(B))_O)$$
induced by the evaluation-at-zero map is an isomorphism.
\end{Pro}

Such a coherent collection in Proposition \ref{coarse disjoint union} is also called $(G,r)$-separated. We still need some preparations before we prove it. For each $g\in G$, define $Y_g$ to be a bounded subset of $P_d(X)$. Then the collection $Y=(Y_g)_{g\in G}$ forms a sequence of bounded subspaces of $X$. Assume that $Y$ satisfies\begin{itemize}
\item[(1)]$g\in Y_g$ for all $g\in Y_g$;
\item[(2)]$Y_g$ is uniformly bounded, i.e. there exists $R>0$ such that $\diam(Y_g)\leq R$ for all $g\in G$.
\end{itemize}
For example, we can take $Y_g=\{g\}$ for each $g\in G$ or $Y_g=B(g,R)$ for some given $R>0$ and all $g\in G$. For the notational convenience, we denote $G_R=(B(g,R))_{g\in G}$ when $Y_g=B(g,R)$. Specially, $G_0=(\{g\})_{g\in G}$.

We denote $\A(O_{g,g})$ to be the subalgebra of $\A(B)$ with all functions whose supports are in $O_{g,g}$. With the notation above, we introduce the following algebra:

\begin{Def}\label{A-algebra}
Define $\IA_{\infty}[Y]$ to be the subalgebra of
\begin{equation}\label{eq algebra}\frac{\prod_{g\in G}C^*(Y_g)\ox\A(O_{g,g})}{\bigoplus_{g\in G}C^*(Y_g)\ox\A(O_{g,g})}\end{equation}
generated by all equivalence class $[(T_g)_{g\in G}]$ of tuples $(T_g)_{g\in G}$ described as follows:\begin{itemize}
\item[(1)]$T_g$ is a bounded function from $(Z_d\times Z_d)\cap(Y_g\times Y_g)$ to $\A(B)\wox\K$ for all $g\in G$ such that
$$\sup_{g\in G}\sup_{x,y\in Z_d}\|T_g(x,y)\|<\infty;$$
\item[(2)]for any bounded subset $B\subset Y_g\cap Z_d$, we have
$$\#\{(x,y)\in B\times B\mid T(x,y)\ne0\}<\infty;$$
\item[(3)]there exists $L>0$ such that
$$\#\{y\in Z_d\mid T_g(x,y)\ne0\}\leq L,\quad\#\{y\in Z_d\mid T_g(y,x)\ne0\}\leq L$$
for all $x\in Y_g\cap Z_d$ and $g\in G$;
\item[(4)]there exists $R\geq 0$ such that $T_g(x,y)=0$ whenever $d(x,y)>R$ for all $g\in G$, the least such $R$ is called the propagation of the representative element $T$, denoted by $\prop(T)$;
\item[(5)]there exists $r>0$ such that for all $x,y\in Y_g\cap Z_d$ and $g\in G$, then
$$\supp(T_g(x,y))\subseteq O_{g,g}$$
for all $n\in\IN$.
\end{itemize}
\end{Def}

The algebraic structure of $\IA_{\infty}[Y]$ is defined by viewing $\IA_{\infty}[Y]$ as the subalgebra of the $C^*$-algebra in equation \eqref{eq algebra} and $A^*_{\infty}(Y)$ is defined to be the norm closure of $\IA_{\infty}[Y]$ with the induced norm.

\begin{Def}
Define $\IA_{L,\infty}[Y]$ to be the $*$-algebra of all bounded, uniformly norm-continuous functions
$$f:\IR_+\to\IA_{\infty}[Y]$$
such that $f=(f_g)$ satisfies the conditions in Definition \ref{twisted localization algebra} with uniform constants.

Define $A^*_{L,\infty}(Y)$ to be the completion of $\IA_{L,\infty}[Y]$ with respect to the norm
$$\|f\|=\sup_{t\in\IR_+}\|f(t)\|.$$
\end{Def}

Let $ev_A$ be the evaluation map
$$ev_A:A^*_{L,\infty}(Y)\to A^*_{\infty}(Y)$$
defined by $ev_A(g)=g(0)$.

\begin{Pro}\label{Roe to A}
Suppose that $O=(O_x)$ is a coherent collection of open subsets of $B$ which is $(G,r)$-separated for some $G$. We denote $G_R=(B(g,R))_{g\in G}$ to be the family of subsets of $P_d(X)$. Then we have\begin{itemize}
\item[(1)]$C^*_{\infty}(P_d(X),\A(B))_O\cong\lim_{R\to\infty}A^*_{\infty}(G_R)$;
\item[(2)]$C^*_{L,\infty}(P_d(X),\A(B))_O\cong\lim_{R\to\infty}A^*_{L,\infty}(G_R)$.
\end{itemize}\end{Pro}

\begin{proof}
Take an arbitary element
$$[T]\in\IC_{\infty}[P_d(X),\A(B)]_O.$$
By definition, we have that
$$\supp(T(x,y))\subset O_x=\bigsqcup_{g\in B(x,l_{d,m})\cap G}O_{x,g},$$
where $x\in K_{d,m+1}\cap U_{d,m}$. Since the coherent open susbet is $(G,r)$-separate. Then we have a direct sum decomposition
$$T(x,y)=\bigoplus_{G\cap B(x,l_{d,m})}T_g(x,y),$$
where
$$T_g(x,y)=T(x,y)|_{O_{x,g}}$$
is the restriction of $T(x,y)$ on $O_{x,g}$ for all $x,y\in Z_d$ and $g\in G$. By property (4) in Definition \ref{twisted Roe algebra}, there exists $r_1>0$ such that
$$\supp(T(x,y))\subseteq B(t_x(x)(s(x)),r_1).$$
Conbining with the fact that
$$O_{x,g}\subseteq B(t_x(g)(s(g)),r),$$
we have that $T_g(x,y)=0$ for all $x$ with $d(t_x(x)(s(x)),t_x(g)(s(g)))>r_1+r$. It follows that there exists $R>0$ such that $T_g(x,y)=0$ for all $d(x,g)>R$.

We define
$$S_{g}(x,y)=(t_{gx})_*(T_g(x,y))$$
for all but finite $g\in G$ and $0$ otherwise. Notice that $T$ has finite propagation and $O$ is $(G,r)$-separated, then $S_g$ is well defined for all $g\in G\cap U_{d,m}$ with $m$ large enough. Then we have that $S_g$ defines an element in $C^*(B(g,R))\ox\A(O_{g,g})$.

Now, notice that the tuples $(S_g)_{g\in G}$ forms an element of $\IA_{\infty}[G_R]$. Then the correspondence $[T]\mapsto[(S_g)]$ extends to a $*$-isomorphism by the fact that the norms in these two $C^*$-algebras agrees. One can see the proof of \cite[Lemma 3.9]{JR2013} for essentially the same arguments which can be used to show the norms in these two algebras agree.
\end{proof}

Now let us recall the notion of strong Lipschitz homotopy introduced by G. Yu in \cite{Yu1997,Yu2000}. Let $(Y_g)_{g\in G}$ and $(W_g)_{g\in G}$ be be two families of uniformly bounded closed subspaces of $P_d(X)$ satisfying the condition in Definition \ref{A-algebra}. A map
$$f:\bigsqcup_{g\in G}Y_g\to\bigsqcup_{g\in G}W_g$$
is said to be Lipschitz if\begin{itemize}
\item[(1)]$f(Y_g)\subset W_g$ for all $g\in G$;
\item[(2)]there exists a constant $c>0$ such that
$$d(f(x),f(y))\leq c\cdot d(x,y)$$
for all $x,y\in Y_g$ and $g\in G$.
\end{itemize}

\begin{Def}[\cite{Yu1997}]
Let $f$ and $h$ be two Lipschitz maps from $\bigsqcup_{g\in G}Y_g$ to $\bigsqcup_{g\in G}W_g$. A continuous homotopy $F(x,t)(t\in [0,1])$ between $f$ and $h$ is said to be strongly Lipschitz if
\begin{itemize}
\item[(1)]$d(F(x,t),F(y,t))\leq Cd(x,y)$ for all $x,y\in Y_g$, $g\in G$ and $t\in[0,1]$, where $C$ is a constant (called the Lipschitz constant of $F$);
\item[(2)]$F$ is equi-continuous in $t$, i.e., for any $\varepsilon>0$, there exists $\delta>0$ such that $d(F(x,t_1),F(x,t_2))\leq\varepsilon$ for all $x\in \bigsqcup_{g\in G}Y_g$ if $|t_1-t_2|<\delta$;
\item[(3)]$F(x,0)=f(x),F(x,1)=h(x)$ for all $x\in \bigsqcup_{g\in G}Y_g$.
\end{itemize}

We say $(Y_g)_{g\in G}$ is strongly Lipschitz homotopy equivalent to $(W_g)_{g\in G}$ if there exists Lipschitz maps
$$f:\bigsqcup_{g\in G}Y_g\to\bigsqcup_{g\in G}W_g$$
and
$$h:\bigsqcup_{g\in G}W_g\to\bigsqcup_{g\in G}Y_g$$
such that $f\circ h$ and $h\circ f$ are strongly Lipschitz homotopy equivalent to the identity maps, respectively.
\end{Def}

The proof of the following result is similar to \cite[Lemma 6.4]{Yu2000}, one can check it with a Eilenberg swindle argument as in \cite{Yu2000}.

\begin{Lem}[\cite{Yu2000}]\label{Yu's lemma}
Let $A^*_{L,0,\infty}(Y)$ be the $C^*$-subalgebra of $A^*_{L,\infty}(Y)$ consisting of those functions $f$ such that $f(0)=0$. We have the following results:\begin{itemize}
\item[(1)]If $Y=(Y_g)_{g\in G}$ is strongly Lipschitz homotopy equivalent to $W=(W_g)_{g\in G}$, then
$$K_*(A^*_{L,0,\infty}(Y))\cong K_*(A^*_{L,0,\infty}(W)).$$
\item[(2)]If $G_R=(B(g,R))_{g\in G}$, then
$$K_*(A^*_{L,0,\infty}(G_R))= 0.$$
\end{itemize}
As a corollary, the evaluation map induces an isomorphism on the level of $K$-theory, i.e.
$$(ev_A)_*:K_*(A^*_{L,\infty}(G_R))\to K_*(A^*_{\infty}(G_R))$$
is an isomorphism.\qed\end{Lem}

\begin{proof}[Proof of Proposition \ref{coarse disjoint union}]
By Proposition \ref{Roe to A}, we have the following commuting diagram:
$$\xymatrix{\lim\limits_{d\to\infty}K_*(C^*_{L,\infty}(P_d(X),\A(B))_O)\ar[r]^{(ev^{\A}_{\infty})_*}\ar[d]_{\cong}&\lim\limits_{d\to\infty}K_*(C^*_{\infty}(P_d(X),\A(B))_O)\ar[d]^{\cong}\\
\lim\limits_{d\to\infty}\lim\limits_{R\to\infty}K_*(A^*_{L,\infty}(G_R))\ar[d]_{\cong}\ar[r]^{(ev_A)_*}&\lim\limits_{d\to\infty}\lim\limits_{R\to\infty}K_*(A^*_{\infty}(G_R))\ar[d]^{\cong}\\
\lim\limits_{R\to\infty}\lim\limits_{d\to\infty}K_*(A^*_{L,\infty}(G_R))\ar[r]^{(ev_A)_*}&\lim\limits_{R\to\infty}\lim\limits_{d\to\infty}K_*(A^*_{\infty}(G_R)).}$$
in which all vertical maps are isomorphisms. By Lemma \ref{Yu's lemma}, we have that the bottom horizontal map $(ev_A)_*$ is an isomorphism. Thus, $(ev_{\infty}^{\A})_*$ is also an isomorphism.
\end{proof}

Finally, we are able to prove Theorem \ref{twisted Baum-Connes}.

\begin{proof}[Proof of Theorem \ref{twisted Baum-Connes}]
For any $r>0$, we define $O_x^{(r)}$ by
$$O_x^{(r)}=\bigcup_{y\in B(x,l_{d,m})}B(t_x(y)(s(y)),r)$$
for any $x\in K_{d,m+1}\cap U_{d,m}$. Then it is easy to check that $O^{(r)}=(O_x^{(r)})_{x\in X}$ is a coherent collection of open subsets.

For any $d\geq 0$, by the definition of the twisted algebras, we have that
$$C^*_{\infty}(P_d(X),\A(B))=\lim_{r\to\infty}C^*_{\infty}(P_d(X),\A(B))_{O^{(r)}},$$
$$C^*_{L,\infty}(P_d(X),\A(B))=\lim_{r\to\infty}C^*_{L,\infty}(P_d(X),\A(B))_{O^{(r)}}.$$
Notice that the limits of $C^*_{\infty}(P_d(X),\A(B))_{O^{(r)}}$ with limits order $\lim_{d\to\infty}\lim_{r\to\infty}$ and $\lim_{d\to\infty}\lim_{r\to\infty}$ are actually the same. Consequently, it suffices to prove that the $K$-theortic homomorphism
$$ev_*:\lim_{d\to\infty}K_*(C^*_{L,\infty}(P_d(X),\A(B))_{O^{(r)}})\to\lim_{d\to\infty}K_*(C^*_{\infty}(P_d(X),\A(B))_{O^{(r)}})$$
induced by the evaluation-at-zero map is an isomorphism for each $r>0$.

For any given $r>0$, as $X$ has bounded geometry, there exists $N_r>0$ such that
$$X=\bigsqcup_{k=1}^{N_r}X_k$$
satisfying that for any $z\in Z_d\cap U_{d,m}$ and $x,y\in X_k\cap B(z,l_{d,m})$, we have
$$\|t_z(x)(s(x))-t_z(y)(s(y))\|>2r.$$
We shall define
$$O^{(r)}_{x,k}=\bigcup_{y\in X_k\cap B(x,l_{d,m})}B(t_x(y)(s(y)),r)$$
for all $x\in X$ and $k\in\{1,\cdots,N_r\}$. Then $O^{(r)}_k=(O^{(r)}_{x,k})_{x\in X}$ is a coherent collection which satisfies the condition in Proposition \ref{coarse disjoint union}. Thus we have that
$$ev_*:\lim_{d\to\infty}K_*(C^*_{L,\infty}(P_d(X),\A(B))_{O^{(r)}_k})\to\lim_{d\to\infty}K_*(C^*_{\infty}(P_d(X),\A(B))_{O^{(r)}_k})$$
for all $k\in\{1,\cdots,N_r\}$. Then it follows from a Mayer-Vietories argument by using Lemma \ref{MV-argument} to complete the proof.
\end{proof}

\section{The geometric Bott map and proof of the main theorem}

In this section, we shall define the Bott map $\beta$ from the $K$-theory group of the Roe algebra to $K$-theory group of the twisted Roe algebra and its localization algebraic version $\beta_L$. The construction relies heavily on the Bott periodicity theorem we introduced in Section 5.

For each $t\geq 1$ and $g\in\S$, denote $g_t\in\S$ to be the function
$$g_t(r)=g(t^{-1}r).$$
For any $x\in X$, we define the Bott map $\beta(x):\S\to\A(B)$ to be
$$(\beta(x))(g)=\beta_{t_x(x)(s(x))}(g),$$
where $t_x$ is the trivilazition on $B(x,l_{d,m})$ defined as in Definition \ref{FCE}, $t_x(x)(s(x))\in B$ and $\beta_{t_x(x)(s(x))}$ is the Bott map defined as in Section 5.

\begin{Def}
For each $d\geq 0$ and $t\in[1,\infty)$, define a map
$$\beta_t:\S\wox\IC_{\infty}[P_d(X)]\to\IC_{\infty}[P_d(X),\A(B)]$$
for each $g\in\S$, $[T]\in\IC_{\infty}[P_d(X)]$ by the formula
$$\beta_t(g\wox[T])=[\beta_t(g\wox T)],$$
where
$$(\beta_t(g\wox T))(x,y)=(\beta(x))(g_t)\wox T(x,y).$$
\end{Def}

\begin{Def}
For each $d\geq 0$ and $t\in[1,\infty)$, define a map
$$(\beta_L)_t:\S\wox\IC_{L,\infty}[P_d(X)]\to\IC_{L,\infty}[P_d(X),\A(B)]$$
for each $g\in\S$, $f\in\IC_{L,\infty}[P_d(X)]$ by the formula
$$((\beta_L)_t(g\wox f))(r)=\beta_t(g\wox f(r)),$$
for each $r\in\IR_+$.
\end{Def}

The following lemma is proved similarly with \cite[Proposition 7.6]{Yu2000} and \cite[Lemma 7.3]{CWY2013}

\begin{Lem}\label{Bott}
For each $d\geq0$, the maps $\beta_t$ and $(\beta_L)_t$ extend respectively to asymptotic morphisms
$$\beta:\S\wox\C^*_{\infty}(P_d(X))\leadsto C^*_{\infty}(P_d(X),A(B))$$
$$\beta_L:\S\wox\C^*_{L,\infty}(P_d(X))\leadsto C^*_{L,\infty}(P_d(X),A(B)).$$
\end{Lem}

\begin{proof}
We will only prove it for $\beta$ and the case for $\beta_L$ follows similarly.

First of all, we claim that
$$\left\|[\beta_t(gh\wox ST)]-[\beta_t(g\wox S)][\beta_t(h\wox T)]\right\|\to0$$
for any $[S],[T]\in\IC_{\infty}[P_d(X)]$. By definition, we can calculate that the $(x,y)$-th entry of the above matrix coefficients of $\beta_t(gh\wox ST)$ and $\beta_t(g\wox S)\beta_t(h\wox T)$ are
$$\sum_{z\in Z_{d}}\beta(x)(g_th_t)\wox S(x,z)T(z,y)$$
and
$$\sum_{z\in Z_{d}}\beta(x)(g_t)(t_{xy})_*(\beta(y)(h_t))\wox S(x,z)T(z,y),$$
respectively. By using \cite[Lemma 3.4]{GWY2008}, it suffices to prove for any $\varepsilon>0$ and $h\in\S$, there exists $T>0$ such that
$$\|\beta(x)(h_t)-(t_{xy})_*(\beta(y)(h_t))\|<\varepsilon$$
for all $t\geq T$. To simplify the notation, we denote $v_{x}=t_x(x)(s(x))$ and $v_y=t_x(y)(s(y))$. Write $R$ the propagation of $[ST]$. Assmue that $d(x,y)\leq R$ and $d(x,x_0)$ is large enough such that $t_y(y)(s(y))=t_{xy}(v_y)$, where $x_0$ is the base point we choose. Combining Definition \ref{translation}, we have that
\begin{equation*}\begin{split}
\|\beta(x)(h_t)-(t_{xy})_*\beta(y)(h_t)\|&=\sup_{(v,r)\in B\times\IR_+}\|h_t(C_{v_x}(v,r))-h_t(C_{v_y}(v,r))\|\\
&=\sup_{(v,r)\in B\times\IR_+}\left\|h\left(t^{-1}\phi(v-v_x),t^{-1}r\right)-h\left(t^{-1}\phi(v-v_y),t^{-1}r\right)\right\|\\
&=\sup_{(v,r)\in B\times\IR_+}\left\|h\left(\phi(t^{-2/p}(v-v_x)),t^{-1}r\right)-h\left(\phi(t^{-2/p}(v-v_y)),t^{-1}r\right)\right\|\\
&=\|\beta_{t^{-2/p}v_x}(h)-\beta_{t^{-2/p}v_y}(h)\|
\end{split}\end{equation*}
Notice that $\|v_x-v_y\|\leq \rho_+(d(x,y))\leq \rho_+(R)$, where $\rho_+$ is the controlled function defined as in Definition \ref{FCE}. By using the fact $\lim\limits_{t\to\infty}t^{-2/p}=0$ and Lemma \ref{cont. Bott}, we complete the proof of the claim.

The rest part of the proof is similar with \cite[Lemma 7.6]{Yu2000}. One can similarly prove that
$$\beta_t(g\wox[T])\leq \|g\|\cdot\|[T]\|.$$
for all $g\in\S$ and $[T]\in\IC_{\infty}[P_d(X)]$. Hence $\beta_t$ extends to an well-defined asymptotic morphism from $\S\wox_{max}C^*_{\infty}(P_d(X))$ to $C^*_{\infty}(P_d(X),\A(B))$. Since $\S$ is nuclear, we complete the proof.
\end{proof}

Note that the asymptotic morphisms
$$\beta_t:\S\wox C^*_{\infty}(P_d(X))\leadsto C^*_{\infty}(P_d(X),\A(B))$$
$$(\beta_L)_t:\S\wox C^*_{L,\infty}(P_d(X))\leadsto C^*_{L,\infty}(P_d(X),\A(B))$$
induce homomorphisms on $K$-theory
$$\beta_*:K_*(\S\wox C^*_{\infty}(P_d(X)))\leadsto K_*(C^*_{\infty}(P_d(X),\A(B)))$$
$$(\beta_L)_*:K_*(\S\wox C^*_{L,\infty}(P_d(X)))\leadsto K_*(C^*_{L,\infty}(P_d(X),\A(B))).$$

\begin{Thm}\label{BottL}
For any $d> 0$, the Bott map
$$(\beta_L)_*:K_*(\S\wox C^*_{L,\infty}(P_d(X)))\leadsto K_*(C^*_{L,\infty}(P_d(X),\A(B))).$$
is an isomorphism.
\end{Thm}

\begin{proof}
The $K$-theory of the localization algebra is invariant under the strong Lipschtiz homotopy equivalence (see \cite{Yu1997,Yu2000}). By a Mayer-Vietoris sequence argument and induction on the dimension of the skeletons \cite{HRY1993,Yu1997}, the general case can be reduced to the zero-dimensional case, i.e., if $\Delta^{(0)}\subseteq P_d(X)$ is the $0$-skeleton of $P_d(X)$, then
$$(\beta_L)_*:K_*\left(\S\wox C^*_{L,\infty}\left(\Delta^{(0)}\right)\right)\to K_*\left(C^*_{L,\infty}\left(\Delta^{(0)},\A(B)\right)\right)$$
is an isomorphism.

Notice that
$$K_*\left(\S\wox C^*_{L,\infty}\left(\Delta^{(0)}\right)\right)=\frac{\prod_{\gamma\in\Delta^{(0)}}K_*(\S\wox C^*_L(\{\gamma\}))}{\bigoplus_{\gamma\in\Delta^{(0)}}K_*(\S\wox C^*_L(\{\gamma\}))}$$
$$K_*\left(C^*_{L,\infty}\left(\Delta^{(0)},\A(B)\right)\right)=\frac{\prod_{\gamma\in\Delta^{(0)}}K_*(C^*_L(\{\gamma\},\A(B)))}{\bigoplus_{\gamma\in\Delta^{(0)}}K_*(C^*_L(\{\gamma\},\A(B)))}.$$
Moreover, one can see that $K_*(\S\wox C^*_L(\{\gamma\}))\cong K_*(\S\wox\K)$ and $K_*(C^*_L(\{\gamma\},\A(B)))\cong K_*(\K\wox\A(B))$ and the Bott map $(\beta_L)_*$ coincides with $\beta_{t_{\gamma}(\gamma)(s(\gamma))}$ when $(\beta_L)_*$ restricts on $K_*(\{\gamma\})\cong K_*(\S)$ to $K_*(C^*_L(\{\gamma\},\A(B)))\cong K_*(\A(B))$. Then the theorem follows from Theorem \ref{K of A(B)} directly.
\end{proof}

Then we can finally prove the main result:

\begin{proof}[Proof of Theorem \ref{main theorem}]
We have the following commuting diagram:
$$\xymatrix{K_*(C^*_{L,\infty}(P_d(X)))\ar[rr]^{(\beta_{L})_*\qquad}\ar[d]_{(ev_{\infty})_*}&&K_*(C^*_{L,\infty}(P_d(X),\A(B)))\ar[d]^{(ev^{\A}_{\infty})_*}\\K_*(C^*_{\infty}(P_d(X)))\ar[rr]^{\beta_*\qquad}&&K_*(C^*_{\infty}(P_d(X),\A(B))).}$$
By Theorem \ref{twisted Baum-Connes} and Theorem \ref{BottL}, we show that $ev_{\infty}^{\A}$ is an isomorphism and $(\beta_L)_*$ is injective. It forces $ev_{\infty}$ to be an injection. Combining Theorem \ref{CNC at inf}, we conclude that the coarse assembly map
$$\mu:\lim_{d\to\infty}K_*(P_d(X))\to K_*(C^*(X))$$
is injective.
\end{proof}

\bibliographystyle{plain}
\bibliography{ref}

\vskip 1cm
\begin{itemize}
\item[] Liang Guo\\
Research Center for Operator Algebras, School of Mathematical Sciences, East China Normal University, Shanghai, 200241, P. R. China.\quad
E-mail: 52205500015@stu.ecnu.edu.cn

\item[] Zheng Luo \\
Research Center for Operator Algebras, School of Mathematical Sciences, East China Normal University, Shanghai, 200241, P. R. China.\quad
E-mail: E-mail: 52195500005@stu.ecnu.edu.cn

\item[] Qin Wang \\
Research Center for Operator Algebras,  and Shanghai Key Laboratory of Pure Mathematics and Mathematical Practice, School of Mathematical Sciences, East China Normal University, Shanghai, 200241, P. R. China. \quad
E-mail: qwang@math.ecnu.edu.cn\\

\item[] Yazhou Zhang \\
Research Center for Operator Algebras, School of Mathematical Sciences, East China Normal University, Shanghai, 200241, P. R. China.\quad
E-mail: 52185500010@stu.ecnu.edu.cn
\end{itemize}
\end{document}